\documentclass{amsart}
\usepackage[utf8]{inputenc}
\usepackage[english]{babel}
\usepackage{amsmath}
\usepackage{amsthm}
\usepackage{amsfonts}
\usepackage{amssymb}
\usepackage{graphicx}
\usepackage{array}

\usepackage{subcaption}

\usepackage{xcolor}

\newtheorem{thm}{Theorem}[section]
\newtheorem{lem}[thm]{Lemma}

\theoremstyle{definition}
\newtheorem{defin}[thm]{Definition}
\newtheorem{ex}[thm]{Example}
\newtheorem{rem}[thm]{Remark}

\title[Periodic stationary solutions of the Nagumo lattice differential]{Periodic stationary solutions of the Nagumo lattice differential equation: existence regions and their number}
\author{Vladim\'{i}r \v{S}v\'{i}gler}
\address{University of West Bohemia, Faculty of Applied Sciences, Department of Mathematics and NTIS, Technick\'{a} 8, 301 00, Pilsen, Czech Republic}
\email{sviglerv@kma.zcu.cz}
\thanks{Author acknowledges the support of the project LO1506 of the Czech Ministry of Education, Youth and Sports under the program NPU I and the support of Grant Agency of the Czech Republic, project no.18-032523S. 
The author is grateful to Petr Stehl\'{i}k and Jon\'{a}\v{s} Volek for their valuable comments and patience. }

\begin{document}

\begin{abstract}
The Nagumo lattice differential equation admits stationary solutions with arbitrary spatial period for sufficiently small diffusion rate. 
The continuation from the stationary solutions of the decoupled system (a system of isolated nodes) is used to determine their types; the solutions are labelled by words from a three-letter alphabet. 
Each stationary solution type can be assigned a parameter region in which the solution can be uniquely identified. 
Numerous symmetries present in the equation cause some of the regions to have identical or similar shape. 
With the help of combinatorial enumeration, we derive formulas determining the number of qualitatively different existence regions. 
We also discuss possible extensions to other systems with more general nonlinear terms and/or spatial structure. 
\end{abstract}

\maketitle 
 
\section{Introduction}
In this paper, we consider the Nagumo lattice differential equation (LDE)
\begin{align}
\label{eq:LDE}
u'_i(t) = d \big(u_{i-1}(t) - 2 u_i(t) + u_{i+1}(t) \big)+ f\big(u_i(t);a\big)
\end{align}
for $i \in \mathbb{Z}, t>0$ with $d>0$, where the nonlinear term $f$ is given by 
\begin{align}
\label{eq:cubic}
f(s;a) = s(1-s)(s-a),
\end{align}
with $a \in (0,1)$. The LDE~\eqref{eq:LDE} is used as a prototype bistable equation arising from the modelling of a nerve impulse propagation in a myelinated axon~\cite{Bell1981}. 
The bistable equations have their use in modelling of active transmission lines~\cite{Nagumo1962, Nagumo1965}, cardiophysiology~\cite{Beeler1977}, neurophysiology~\cite{Bell1981}, nonlinear optics~\cite{Kevrekidis2009}, population dynamics~\cite{Levin1974} and other fields. 

Throughout this paper, we shall use correspondence of the LDE~\eqref{eq:LDE} and the Nagumo graph differential equation on a cycle~\eqref{eq:GDE}. 
The graph and lattice reaction-diffusion differential equations are used in modelling of dynamical systems whose spatial structure is not continuous but can be described by individual vertices (possibly infinitely many) and their interactions via edges. 
The main difference is such that a lattice (the underlying structure of~\eqref{eq:LDE}) is infinite but there are strong assumptions on its regularity whereas graphs are usually (but not exclusively) finite and nothing is assumed about their structure in general. 
Such models arise in population dynamics~\cite{Allen1987}, image processing~\cite{Lindeberg1990}, chemistry~\cite{Laplante1992}, epidemiology~\cite{Kiss2017} and other fields. 
Alternative focus lies in the numerical analysis where the graph differential equations describe spatial discretizations of partial differential equations~\cite{Hosek2019,Hupkes2015}. 
Mathematically, the interaction between analytic and graph theoretic properties represent new and interesting challenges.
The graph and lattice reaction-diffusion differential equations exhibit behaviour which can not be observed in their partial differential equation counterparts such as a rich structure of stationary solutions~\cite{Stehlik2017}, or other phenomena described in the forthcoming text such as pinning, multichromatic waves and other.

The LDE~\eqref{eq:LDE} is known to possess travelling wave solutions of the form
\begin{align}
\begin{split}
\label{eq:disWave}
u_i(t)  &= \varphi(i-ct), \\
\lim_{s \to -\infty} \varphi(s) = 0, &  \quad \lim_{s \to +\infty} \varphi(s) = 1.
\end{split}
\end{align}
As the authors in~\cite{Keener1987} and~\cite{Zinner1992} had shown, there are nontrivial parameter $(a,d)$-regimes preventing the solutions of type~\eqref{eq:disWave} from travelling ($c = 0$) creating the so-called \textit{pinning region}. 
This propagation failure phenomenon can be partially clarified by the existence of countably many stable stationary solutions (including the periodic ones) of~\eqref{eq:LDE} which inhabit mainly the pinning region, see Figure~\ref{fig:regions}. 
\begin{figure}[ht]
\centering
\includegraphics[width = \textwidth]{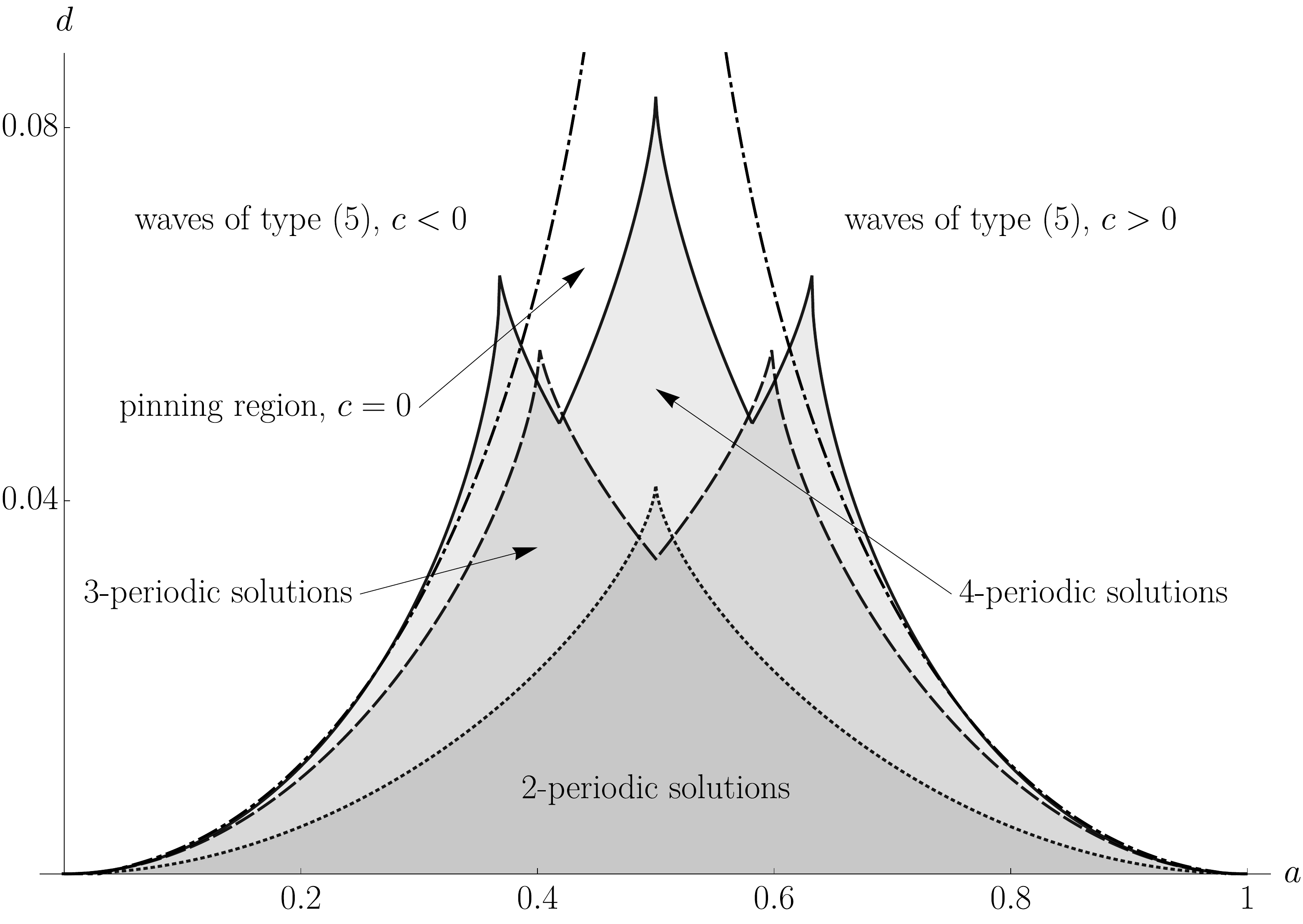}
\caption{Numerically computed regions in the $(a,d)$-plane in which the waves of the type~\eqref{eq:disWave} travel (the regions above the two dot-dashed curves) and the pinning region (the region between the $a$-axis and the two dot-dashed curves). 
To better illustrate the significance and the presence of the stable heterogeneous $n$-periodic stationary solutions of the LDE~\eqref{eq:LDE} in the pinning region, we include the existence regions for the two-periodic stable stationary solutions (dotted edge), the three-periodic stable stationary solutions (dashed edge) and the four-periodic stable stationary solutions (solid edge).}
\label{fig:regions}
\end{figure}
This pinning phenomenon was observed in other lattice systems~\cite{Vleck1998}, experimentally in chemistry~\cite{Laplante1992} and also hinted in systems of coupled oscillators~\cite{Bramburger2018}. 
It is worth mentioning that the equation~\eqref{eq:LDE} can be obtained via spatial discretization of the Nagumo partial differential equation
\begin{align*}
u_t(x,t) = d u_{xx}(x,t) + f\big( u(x,t); a \big),
\end{align*}
which possesses travelling wave solutions of type 
\begin{align}
\begin{split}
\label{eq:ctsWave}
u(x,t)  &= \varphi(x-ct), \\
\lim_{s \to -\infty} \varphi(s) = 0, &  \quad \lim_{s \to +\infty} \varphi(s) = 1;
\end{split}
\end{align}
the waves are pinned if and only if $\int_0^1 f(s;a) \, \mathrm{d}s = 0$. 

The waves of type~\eqref{eq:disWave} (whether the travelling or the pinned ones) can be perceived as solutions connecting two homogeneous stable states of the LDE~\eqref{eq:LDE}; constant $0$ and constant $1$. 
This concept can be generalized to the solutions connecting the nonhomogeneous periodic steady states. 
Let $\mathrm{u}, \mathrm{v} \in \mathbb{R}^n$ be two vectors such that their periodic extensions are asymptotically stable stationary solutions of the LDE~\eqref{eq:LDE}. 
The multichromatic wave is then a solution of a form
\begin{align}
\label{eq:mulWave}
\begin{split}
u_i(t)  &= \phi(i-ct), \\
\lim_{s \to -\infty} \phi(s) = \mathrm{u}, &  \quad \lim_{s \to +\infty} \phi(s) = \mathrm{v},
\end{split}
\end{align}
where 
\begin{align*}
\phi = (\phi_1, \phi_2, \ldots, \phi_n): \mathbb{R} \to \mathbb{R}^n. 
\end{align*}
The bichromatic waves connecting homogeneous and two-periodic solutions were examined in~\cite{Hupkes2019a}. 
The tri- and quadrichromatic waves incorporating three- and four-periodic solutions were studied in detail in~\cite{Hupkes2019b}. 
Stationary solutions with analogous construction idea, the \textit{oscillatory plateaus} whose limits approach homogeneous steady states and there exists a middle section close to a periodic stationary solution, were analysed in~\cite{Bramburger2019}. 

Motivated by the importance of detailed understanding of the existence of the stationary solutions to the analysis of the advanced structures, the focal point of this paper is the examination of the $(a,d)$-regions in which particular periodic stationary solutions of the LDE~\eqref{eq:LDE} exist. 
Our aim is to derive counting formulas for inequivalent existence regions; the notion of equivalence is rigorously defined in the forthcoming section since it requires certain technical preliminaries. 
It is useful to have a detailed knowledge of the shape of the regions because of their connection to other phenomena. 
It has been shown in~\cite{Hupkes2019a, Hupkes2019b} that they are closely related to the travelling regions of the multichromatic waves. 
As simulations hint (see Figure~\ref{fig:regions}), the regions corresponding to the stable stationary solutions seem to inhabit mainly the pinning region. 
Finally, we emphasize their obvious significance as the condition for emergence of spatial patterns in the LDE~\eqref{eq:LDE}. 
To reach the goal, we employ the idea from~\cite{Hupkes2019c} where we have shown a one-to-one correspondence of the LDE~\eqref{eq:LDE} $n$-periodic stationary solutions and stationary solutions of the Nagumo graph differential equation (GDE) on an $n$-vertex cycle 
\begin{align}
\label{eq:GDE}
\begin{cases}
\mathrm{u}'_1(t) =& d\big(\mathrm{u}_{n}(t)-2 \mathrm{u}_1(t) + \mathrm{u}_2(t)\big) + f\big(\mathrm{u}_1(t);a\big) , \\
\mathrm{u}'_2(t) =& d\big(\mathrm{u}_{1}(t)-2 \mathrm{u}_2(t) + \mathrm{u}_3(t)\big) + f\big(\mathrm{u}_2(t);a\big) , \\
&\vdots \\ 
\mathrm{u}'_i(t) =& d\big(\mathrm{u}_{i-1}(t)-2 \mathrm{u}_i(t) + \mathrm{u}_{i+1}(t)\big) + f\big(\mathrm{u}_i(t);a\big) , \\
&\vdots \\ 
\mathrm{u}'_n(t) =& d\big(\mathrm{u}_{n-1}(t)-2 \mathrm{u}_n(t) + \mathrm{u}_{1}(t)\big) + f\big(\mathrm{u}_n(t);a\big) ,
\end{cases}
\end{align}
and, subsequently, with vectors of length $n$ having elements in the three letter alphabet $\mathcal{A}_3 = \{ \mathfrak{0,a,1} \}$, also called the \textit{words}. 
The words encode the origin of the bifurcation branches for $d=0$ whose existence can be shown by using the implicit function theorem for $d>0$ small enough. 
Moreover, the implicit function theorem also implies that the solutions preserve their stability and the asymptotically stable solutions can be thus identified with words created with the two letter alphabet $\mathcal{A}_2 = \{ \mathfrak{0,1} \}$. 
The region in the $(a,d)$-space belonging to a solution labelled by a word $\mathrm{w}$ is denoted by $\Omega_\mathrm{w} \subset \mathcal{H} = [0,1] \times \mathbb{R}^+$.     
Since the stationary problem for~\eqref{eq:GDE} is equivalent to the problem of searching for the roots of a $3^n$-th order polynomial it is a convoluted task to derive some information about the regions. 
There are known lower estimates for their upper boundaries,~\cite{Cheng2005}, asymptotics near threshold points $a \approx 0$, $a \approx 1$ and numerical results, both~\cite{Hupkes2019a, Hupkes2019b}. 
The computations and the numerical simulations can be cumbersome to carry out and thus the exploitation of the equation symmetries is beneficial. 
The idea is to observe, when a symmetry present in the equation relates two regions $\Omega_\mathrm{w}$ without any a-priori knowledge of their shapes. 
For example, the LDE~\eqref{eq:LDE} is invariant to an index shift and the GDE~\eqref{eq:GDE} is invariant to the rotation of indices. 
Consider $n=3$, then given a parameter tuple $(a,d) \in \mathcal{H}$, if there exists a stationary solution $\mathrm{u}_1$ of the GDE~\eqref{eq:GDE} emerging from $(0,0,1)$ for $d=0$, then there surely exist solutions $\mathrm{u_2}, \mathrm{u_3}$ emerging from $(0,1,0), (1,0,0)$, respectively. 
Moreover, $\mathrm{u}_1, \mathrm{u}_2, \mathrm{u}_3$ have identical values which are just rotated by one element to the left. 
We can thus say that the regions of existence of the solutions emerging from $(0,0,1), (0,1,0)$ and $(1,0,0)$ are identical, i.e., $\Omega_\mathfrak{001} = \Omega_\mathfrak{010} = \Omega_\mathfrak{100}$. 

We show how the symmetries of the LDE~\eqref{eq:LDE} and the GDE~\eqref{eq:GDE} correspond and how they propagate to the set of the labelling words $\mathcal{A}_3^n$. 
Namely, the index rotation $i \mapsto i+1$, the reflection $i \mapsto n-i+1$ create word subsets whose respective regions are identical. 
The value switch $\mathfrak{0} \leftrightarrow \mathfrak{1}$ relates solution types whose respective regions are axially symmetric to each other. 
To this end, we define groups acting on the set of the words $\mathcal{A}_3^n$ and compute the number of their orbits (the number of the word subsets which are pairwise unreachable by the action of the group) via Burnside's lemma, Theorem~\ref{lem:burn}. 
We next restrict the computations to the words whose primitive period is equal to their length since the periodic extension of the GDE~\eqref{eq:GDE} stationary solution of a certain type (e.g., $\mathfrak{0a0a0a}$) is identical to a periodic extension of its subword with the length equal to the original word's primitive period ($\mathfrak{0a}$ here). 
The main tool is M\"{o}bius inversion formula in this case, Theorem~\ref{thm:mobius}. 
The division of the word set $\mathcal{A}_3^n$ into orbits with respect to the action of a group can be achieved with the cost proportional to the number of the words ($3^n$ in this case), see~\cite{Butler1991}.
Our results do not help with the generation of the representative words directly but enable us to easily determine their number.
All results are also provided for asymptotically stable stationary solutions of the LDE~\eqref{eq:LDE} whose corresponding labelling set is $\mathcal{A}_2^n$. 

The paper is organized as follows. 
In \S2 we provide an overview of the properties of the periodic stationary solutions of the LDE~\eqref{eq:LDE} including the introduction of the labelling scheme and the statement of our main result, Theorem~\ref{thm:main}. 
We next include a list of relevant symmetries of the equation and their influence on the regions $\Omega_\mathrm{w}$ and conclude with presentation of the used group-theoretical tools together with the commentary of the known results. 
Using the formal definitions from the preceding text, \S3 is devoted to the derivation of lemmas needed for the proof of the main statement in \S4. 
The final paragraphs elaborate on possible extensions to other models and we discuss open questions therein. 

\section{Preliminaries}

\subsection{Periodic stationary solutions and existence regions}
\label{ssec:prel:ss}

Searching for a general stationary solution of the LDE~\eqref{eq:LDE} requires solving a countable system of nonlinear analytic equations. 
The restriction to periodic solutions simplifies the case to a finite-dimensional problem. 
Indeed, the problem is thus reduced to finding stationary solutions of the GDE~\eqref{eq:GDE}.

\begin{lem}[{\cite[Lemma 1]{Hupkes2019c}}]
\label{lem:equiv}
Let $n \geq 3$. The vector $\mathrm{u} =  (\mathrm{u}_1, \mathrm{u}_2, \ldots, \mathrm{u}_n )$ is a stationary solution of the GDE~\eqref{eq:GDE} if and only if its periodic extension $u$ is an $n$-periodic stationary solution of the LDE~\eqref{eq:LDE}. 
Moreover, $\mathrm{u}$ is an asymptotically stable solution of the GDE~\eqref{eq:GDE} if and only if $u$ is an asymptotically stable solution of the LDE~\eqref{eq:LDE} with respect to the $\ell^\infty$-norm. 
\end{lem} 
If $\mathrm{u} = (\mathrm{u}_1, \mathrm{u}_2, \ldots, \mathrm{u}_n) \in \mathbb{R}^n$ is a vector then the periodic extension $(u_i)_{i \in \mathbb{Z}} \in \ell^\infty$ of $\mathrm{u}$ satisfies $u_i = \mathrm{u}_{1+\text{mod} \, (i,n)}$ for all $i \in \mathbb{Z}$.
In the further text, the function $\text{mod}\,(a,b)$ denotes the remainder of the integer division of $a/b$  for $a,b \in \mathbb{N}$. 

Let us denote the function on the right-hand side of the GDE~\eqref{eq:GDE} by $h \colon \mathbb{R}^n \times (0,1) \times \mathbb{R}_0^+ \to \mathbb{R}^n$, 
\begin{align}
\label{eq:GDErhs}
h(\mathrm{u};a,d) =\left( 
\begin{array}{c}
d \big(\mathrm{u}_{n}-2 \mathrm{u}_1 + \mathrm{u}_2\big) + f(\mathrm{u}_1;a)  \\
\vdots \\ 
d\big(\mathrm{u}_{i-1}-2 \mathrm{u}_i + \mathrm{u}_{i+1}\big) + f(\mathrm{u}_i;a)  \\
\vdots \\ 
d\big(\mathrm{u}_{n-1}-2 \mathrm{u}_n + \mathrm{u}_{1}\big) + f(\mathrm{u}_n;a) 
\end{array}
\right).
\end{align}
The problem of finding a stationary solution of the GDE~\eqref{eq:GDE} can be now reformulated as 
\begin{align}
\label{eq:GDEstat}
h(\mathrm{u};a,d) = 0. 
\end{align}
The problems of type~\eqref{eq:GDEstat}, i.e., a diagonal nonlinear perturbation of a finite-dimen\-sional linear operator, are being treated with a wide spectrum of methods ranging from variational techniques, topological approaches to monotone operator theory, see~\cite{Volek2016} and references therein. 
We derive some information about the system using the perturbation theory. 
Suppose $d=0$, then the problem
\begin{align}
\label{eq:GDEstatNoDif}
h(\mathrm{u};a,0) = 0
\end{align}
has precisely $3^n$ solutions $\mathrm{u} \in \mathbb{R}^n$ which are vectors of length $n$ with the coordinates in the set $\{ 0,a,1 \}$; the system~\eqref{eq:GDEstatNoDif} contains $n$ independent equations. 
There is also an easy way to determine the stability of the roots of~\eqref{eq:GDEstatNoDif}. 
One can readily calculate that 
\begin{align*}
f'(0;a) = -a, \qquad f'(a;a) = a(1-a), \qquad f'(1;a) = a-1,
\end{align*}
which gives $f'(s;a) <0$ for either $s=0$ or $s=1$ and $f'(s;a)>0$ for $s = a$.
The derivative of the function $h$ with respect to the first variable, $D_1 h(\mathrm{u};a,0)$, is a regular diagonal matrix at each solution of~\eqref{eq:GDEstatNoDif}
\begin{align*}
D_1 h(\mathrm{u};a,0) = \mathrm{diag} \, \big(
f'(\mathrm{u}_1; a), f'(\mathrm{u}_2; a), \ldots, f'(\mathrm{u}_n;a) \big).
\end{align*}
If the solution vector contains the value $a$ then it is an unstable stationary solution of the GDE~\eqref{eq:GDE} and it is stable otherwise for $d=0$.
Let some $a^*\in(0,1)$ be given. 
The implicit function theorem now ensures the existence of the solutions of the system~\eqref{eq:GDEstat} for $(a,d) \in \mathcal{U} \cap \mathcal{H}$, where $\mathcal{U}$ is some neighbourhood of the point $(a^*,0)$. 
The parameter dependence is smooth and the sign of the Jacobian is preserved. 

The discussion above justifies the introduction of the naming scheme for the roots of~\eqref{eq:GDEstat} where each solution is identified with the origin of its bifurcation branch at $d=0$. 
It is important to realize that the parameter $a \in (0,1)$ is allowed to vary in our considerations. 
The identification must be made through the substitute alphabet $\mathcal{A}_3 = \{ \mathfrak{0,a,1} \}$ and we define a function $\mathrm{w}_{|a} \colon \mathcal{A}_3^n \to \{ 0,a,1 \}^n$ for given $a \in [0,1]$ by
\begin{align*}
(\mathrm{w}_{|a})_i = 
\begin{cases}
0, & \mathrm{w}_i = \mathfrak{0}, \\ 
a, & \mathrm{w}_i = \mathfrak{a}, \\
1, & \mathrm{w}_i = \mathfrak{1}. 
\end{cases}
\end{align*}

\begin{defin}[{\cite[Definition 2.1]{Hupkes2019b}}]
\label{def:sol-type}
Consider a word $\mathrm{w} \in \mathcal{A}_3^n$ together with a triplet
\begin{align*}
(\mathrm{u},a,d) \in [0,1]^n \times (0,1) \times \mathbb{R}_0^+. 
\end{align*}
Then we say that $\mathrm{u}$ is an equilibrium of the type $\mathrm{w}$ if there exists a $C^1$-smooth curve 
\begin{align*}
[0,1] \ni t \mapsto \big( \mathrm{v}(t), \alpha(t), \delta(t) \big) \in [0,1]^n \times (0,1) \times \mathbb{R}_0^+
\end{align*}
so that we have 
\begin{align*}
&(\mathrm{v}, \alpha, \delta)(0) = (\mathrm{w}_{|a},a,0), \\
&(\mathrm{v}, \alpha, \delta)(1) = (\mathrm{u},a,d), 
\end{align*}
together with 
\begin{align*}
h(\mathrm{v}(t);\alpha(t),\delta(t)) = 0, \qquad \mathrm{det}\, D_1 h(\mathrm{v}(t);\alpha(t),\delta(t)) \neq 0
\end{align*}
for all $0\leq t \leq 1$. 
\end{defin}

We define an open pathwise connected set for each $\mathrm{w} \in \mathcal{A}_3^n$ by
\begin{align*}
\Omega_\mathrm{w} = \big\{ (a,d) \in \mathcal{H} \, | \, \text{the system~\eqref{eq:GDEstat} admits a solution of the type w} \big\}. 
\end{align*} 
Under further considerations, it can be shown that any parameter-dependent solution $u_\mathrm{w}(a,d)$ of type $\mathrm{w}$ of the system~\eqref{eq:GDEstat} is uniquely defined in $\Omega_\mathrm{w}$ and if $(a,d) \in \Omega_{\mathrm{w}_1} \, \cap \, \Omega_{\mathrm{w}_2}\neq\emptyset$ for any two given words $\mathrm{w}_1 \neq \mathrm{w}_2$ then $u_{\mathrm{w}_1}(a,d) \neq u_{\mathrm{w}_2}(a,d)$.
We recommend the reader to consult~\cite[\S2.1]{Hupkes2019b} for a full-length discussion. 
The notion of solution type can be now passed on to the periodic stationary solutions of the LDE~\eqref{eq:LDE} via the statement of Lemma~\ref{lem:equiv}, see 
Figure~\ref{fig:ss} for illustration. 

\begin{figure}[ht]
\centering
\begin{subfigure}[b]{.43\textwidth}
	
	\begin{subfigure}[b]{\textwidth}
	\includegraphics[width = \textwidth]{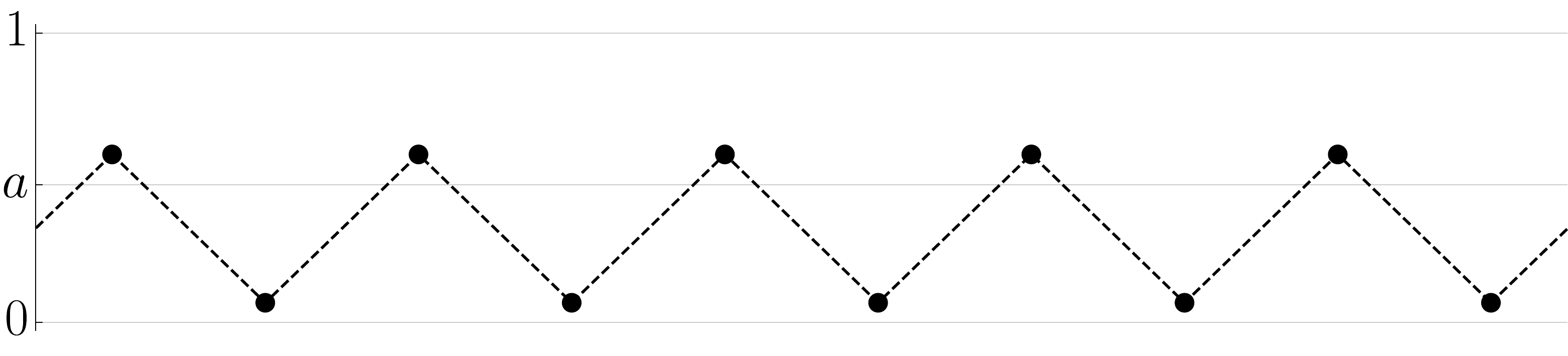}
	\caption{A two-periodic stationary solution of type $\mathfrak{0a}$. }
	\end{subfigure}

	\begin{subfigure}[b]{\textwidth}
	\includegraphics[width = \textwidth]{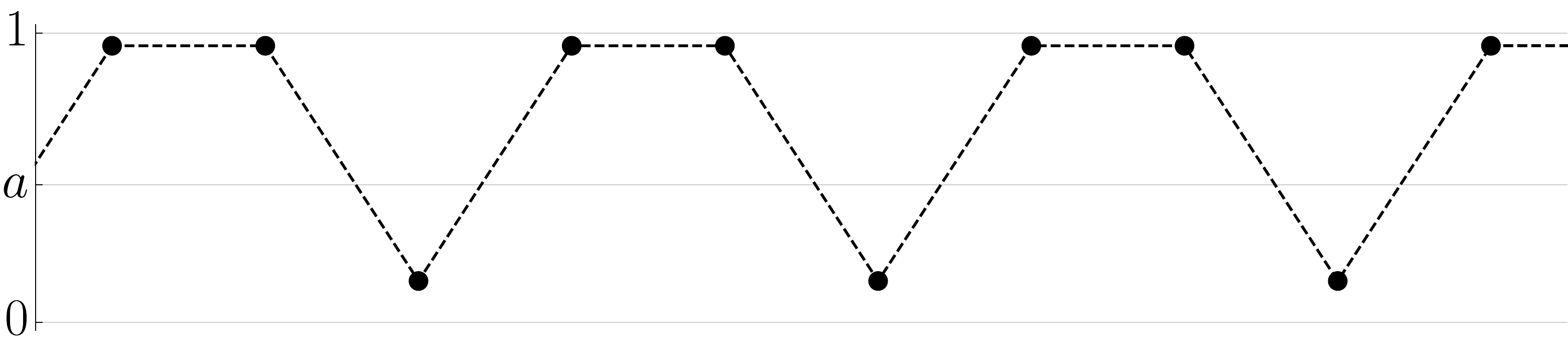}
	\caption{A three-periodic stationary solution of type $\mathfrak{011}$. }
	\end{subfigure}

	\begin{subfigure}[b]{\textwidth}
	\includegraphics[width = \textwidth]{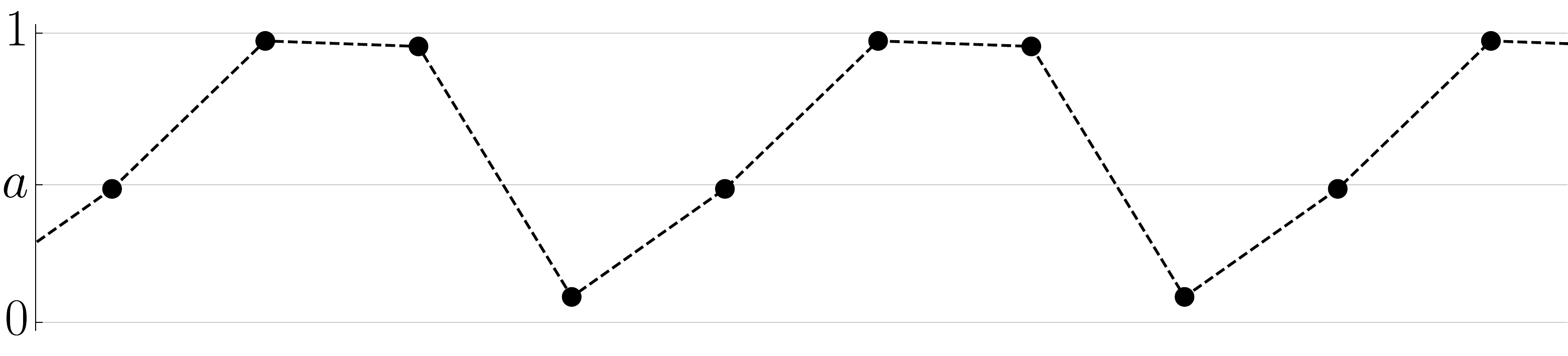}
	\caption{A four-periodic stationary solution of type $\mathfrak{0a11}$. }
	\end{subfigure}	
\end{subfigure}~
\begin{subfigure}[b]{.53\textwidth}
\includegraphics[width=\textwidth]{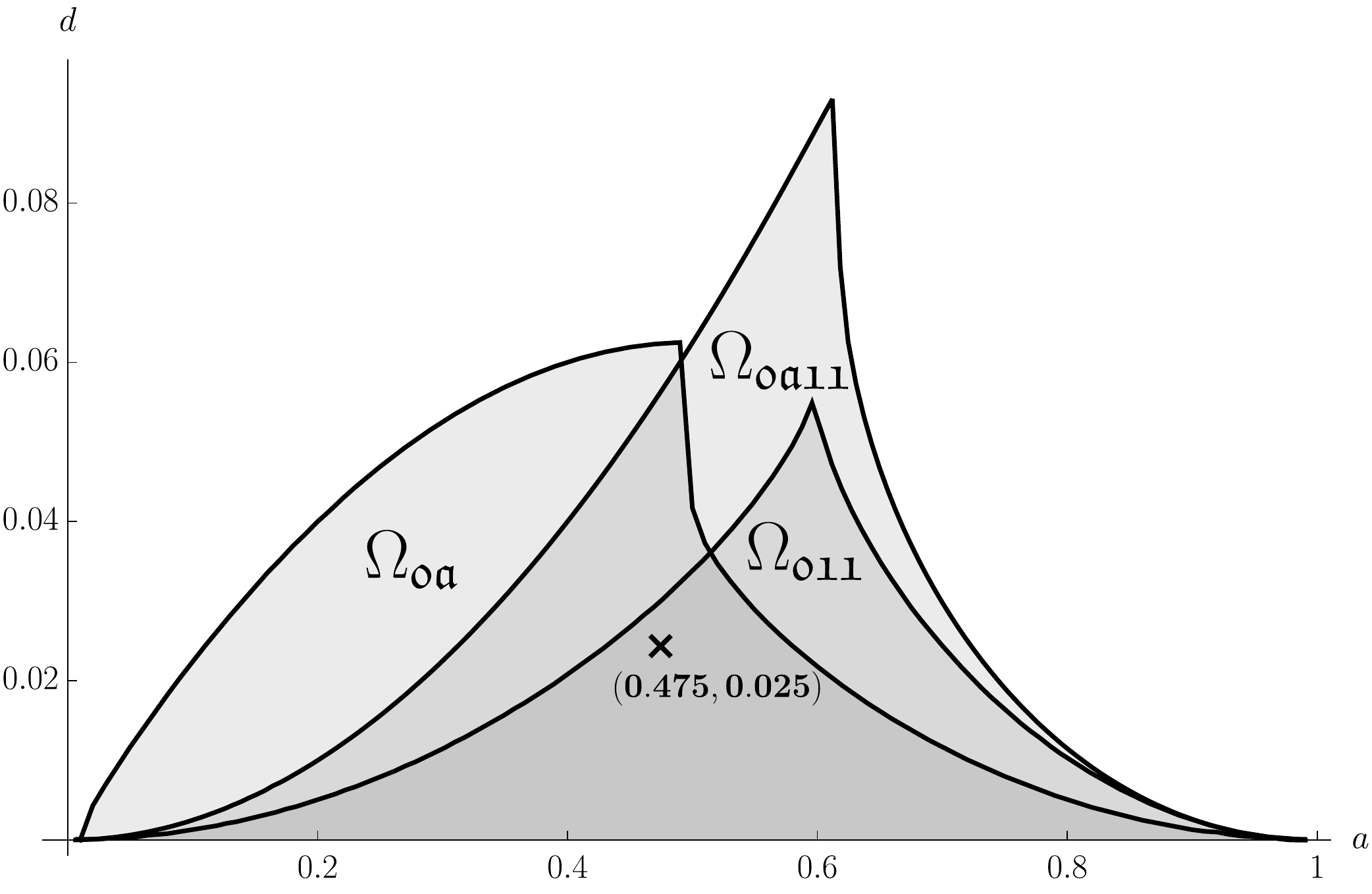}
\caption{Regions of existence $\Omega_\mathrm{w}$ for solutions of type $\mathfrak{0a}$, $\mathfrak{011}$, $\mathfrak{0a11}$.}
\end{subfigure}	

\caption{Examples of two-, three- and four-periodic stationary solutions of the LDE~\eqref{eq:LDE} and the regions of existence for solutions of their respective type.
The parameters $(a,d)=(0.475,0.025)$ are set to be identical in all three cases.}
\label{fig:ss}
\end{figure}

\begin{rem}
The definition of the naming scheme, Definition~\ref{def:sol-type}, ensures that a solution $\mathrm{u}_\mathrm{w}$ of a given type $\mathrm{w} \in \mathcal{A}_3^n$ preserves its stability inside $\Omega_\mathrm{w}$ since the determinant of the Jacobian matrix is not allowed to change its sign. 
Words from $\mathcal{A}_2^n = \{ \mathfrak{0,1} \}^n$ thus represent asymptotically stable steady states.  
\end{rem}

\subsection{Symmetries of the periodic solutions}
\label{sec:syms}
We start with a list of symmetries of the system~\eqref{eq:GDEstat} which are relevant to the similarities of the regions $\Omega_\mathrm{w}$ and then discuss their general impact on the number of the regions. 
Note that the results apply to the periodic stationary solutions of the LDE~\eqref{eq:LDE} via Lemma~\ref{lem:equiv}.

\subsubsection{Rotations}
Let the rotation operator $r: \mathcal{A}_3^n \to \mathcal{A}_3^n$ be defined by 
\begin{align}
\label{eq:rot}
\big( r (\mathrm{w}) \big)_i = \mathrm{w}_{1+\text{mod}\,(i,n)}
\end{align}
for $i = 1, 2, \ldots, n$ with obvious extension to vectors in $[0,1]^n$. 
A shift in indexing in~\eqref{eq:GDEstat} shows that $\mathrm{u} \in [0,1]^n$ is the system~\eqref{eq:GDEstat} solution of type $\mathrm{w} \in \mathcal{A}_3^n$ if and only if $r(\mathrm{u})$ is a solution of type $r(\mathrm{w})$.
Note that this is true in general even if $\mathrm{u}$ cannot be assigned a type; the claim ``$\mathrm{u}$ is a solution of the system~\eqref{eq:GDEstat}'' is invariant with respect to the rotation $r$. 
As a direct consequence, we have 
\begin{align*}
\Omega_\mathrm{w} = \Omega_{r(\mathrm{w})}
\end{align*}
for all $\mathrm{w} \in \mathcal{A}_3^n$. 

The transformation $r$ generates a finite cyclic group of order $n$ which we denote by 
\begin{align*}
C_n = \big( \{ r^0, r^1, \ldots, r^{n-1}\}, \, \circ \, \big). 
\end{align*}
where the group operation $\circ$ is composition of the rotations $r^i \, \circ \, r^j = r^{\text{mod} \, (i+j,n)}$. 
For the sake of consistency with the future notation, we denote the identity element $e$ by $r^0$ and $r^1=r$. 
Let us mention one fact which is implicitly used throughout the paper. 
If $i$ and $n$ are relatively coprime, then $r^i$ is a generator of the group $C_n$. 
For example, let $n=4$, then the repetitive composition of $r^3$ gives the sequence $r^3 \to r^2 \to r^1 \to r^0 \to r^3 \to \ldots$ which covers the whole element set of $C_4$. 
On the other hand, the composition of $r^2$ gives $r^2 \to r^0 \to r^2 \to \ldots$ which does not span the whole element set of $C_4$. 

\subsubsection{Reflections}
Let the reflection operator $s: \mathcal{A}_3^n \to \mathcal{A}_3^n$ be defined by
\begin{align}
\label{eq:refl}
\big( s (\mathrm{w}) \big)_i = \mathrm{w}_{n-i+1}
\end{align}
together with its natural extension to vectors in $[0,1]^n$. 
Similar argumentation as in the previous paragraph shows that $\mathrm{u} \in [0,1]^n$ is a solution of type $\mathrm{w}$ of the system~\eqref{eq:GDEstat} if and only if $s(\mathrm{u})$ is a solution of type $s(\mathrm{w})$.

Adding the reflection $s$ to the cyclic group $C_n$ results in construction of the dihedral group $D_n$ which is generated by the transformations $r$ and $s$. 
Let us denote the composition of the rotation $r^i$ and the reflection $s$ by $sr^i = r^i \circ s$ (i.e., we first reflect and then rotate).   
For the sake of consistency, we also set $sr^0 = s$. 
This allows us to define the dihedral group
\begin{align*}
D_n = \big( \{ r^0, r^1, \ldots, r^{n-1}, sr^0, sr^1, \ldots, sr^{n-1} \}, \, \circ \, \big)  
\end{align*} 
and 
\begin{align*}
\Omega_\mathrm{w} = \Omega_{g(\mathrm{w})}
\end{align*}
holds for all $\mathrm{w} \in \mathcal{A}_3^n$ and $g \in D_n$. 

\subsubsection{Value permutation}
The third symmetry exploits a specific property of the cubic nonlinearity 
\begin{align*}
f(s;a) = -f(1-s;1-a)
\end{align*}
with $s,a \in [0,1]$.
We therefore have
\begin{align}
\label{eq:perm_sym}
h(\mathrm{u};a,d) = -h(\mathbf{1}-\mathrm{u};1-a,d)
\end{align}
for any $\mathrm{u} \in [0,1]^n$ and $a \in [0,1]$ where the subtraction $\mathbf{1}-\mathrm{u}$ is element-wise.  
Let us define the value permutation $\pi: \mathcal{A}_3^n \to \mathcal{A}_3^n$ by
\begin{align}
\label{eq:perm}
\big(\pi (\mathrm{w}) \big)_i = 
\begin{cases}
\mathfrak{1}, & \mathrm{w}_i = \mathfrak{0}, \\
\mathfrak{a}, & \mathrm{w}_i = \mathfrak{a}, \\
\mathfrak{0}, & \mathrm{w}_i = \mathfrak{1}.
\end{cases}
\end{align}
The equality~\eqref{eq:perm_sym} now shows that $\mathrm{u}$ is a solution of type $\mathrm{w}$ of the system~\eqref{eq:GDEstat} if and only if $\mathbf{1}-\mathrm{u}$ is a solution of type $\pi(\mathrm{w})$ with $a \mapsto 1-a$. 
As a direct consequence, 
\begin{align*}
\Omega_\mathrm{w} = \mathcal{T} \big(\Omega_{\pi(\mathrm{w})}\big)
\end{align*}
holds for all $\mathrm{w} \in \mathcal{A}_3^n$ where $\mathcal{T} \colon \mathcal{H} \to \mathcal{H}$ is 
\begin{align}
\label{eq:T}
\mathcal{T}(a,d) = (1-a,d).
\end{align}
The transformation $\mathcal{T}$ is a vertical reflection with respect to the line $a=1/2$.  

The operation $\pi$ generates the two element group 
\begin{align*}
\Pi = \big(\{ e, \pi \}, \, \circ \, \big) , 
\end{align*}
where $e$ is the identity element.  
The group $\Pi$ can be also restricted to operate on the set of all words made with the two letter alphabet $\mathcal{A}_2$ by 
\begin{align*}
\big(\pi (\mathrm{w}) \big)_i = 
\begin{cases}
\mathfrak{1}, & \mathrm{w}_i = \mathfrak{0}, \\
\mathfrak{0}, & \mathrm{w}_i = \mathfrak{1}.
\end{cases}
\end{align*}
To enlighten the notation, we denote the symbol permutation group by the letter $\Pi$ regardless of the used alphabet.

In virtue of the previous notation, let us define $\pi r^i = r^i \, \circ \, \pi$ and $\pi s r^i = r^i \, \circ \, s \, \circ \, \pi$ and the group $C_n^\Pi$ by
\begin{align*}
C_n^\Pi = \big( \{ r^0, r^1, \ldots, r^{n-1}, \pi r^0, \pi r^1, \ldots, \pi r^{n-1} \} , \, \circ \, \big). 
\end{align*}
Note that the group $C_n^\Pi$ contains elements from $C_n$ and the elements from $C_n$ composed with the symbol permutation $\pi$. 
Equivalently, we define the group $D_n^\Pi$ by 
\begin{align*}
D_n^\Pi = \left( \left\{ 
\begin{array}{l}
 r^0, r^1, \ldots, r^{n-1}, \pi r^0, \pi r^1, \ldots, \pi r^{n-1}, \\
 sr^0, sr^1, \ldots, sr^{n-1}, \pi s r^0, \pi s r^1, \ldots, \pi sr^{n-1}
\end{array}
\right\} , \, \circ \, \right). 
\end{align*}
Although our main aim is the examination of the action of the group $D_n^\Pi$ it is convenient to study the group $C_n^\Pi$ separately to be able to obtain partial results which are used in the proof of the main theorem. 
Let us also emphasize that the action of the groups $C_n^\Pi$ and $D_n^\Pi$ preserves stability of the corresponding solutions. 

\subsubsection{Primitive periods}
Let us assume that a word $\mathrm{w}$ of length $n$ has a primitive period of length $m<n$ (say, $\mathfrak{1aa1aa}$), i.e., it consists of $n/m$-times repeated word $\mathrm{w}_m$ of length $m$ ($\mathfrak{1aa}$ in this case). 
Then surely 
\begin{align*}
\Omega_\mathrm{w} = \Omega_{\mathrm{w}_m};
\end{align*}
their respective regions are identical. 
It is not difficult to include this in the counting formulas alone  but the interplay with the group operations ($C_n, D_n, C_n^\Pi, D_n^\Pi$) is more intricate and is treated later via M\"{o}bius inversion formula, Theorem~\ref{thm:mobius}. 

\subsubsection{Other solution properties}
It is clear that regions belonging to the constant solutions of type $\mathfrak{00}\ldots\mathfrak{0}$, $\mathfrak{aa}\ldots\mathfrak{a}$ and $\mathfrak{11}\ldots\mathfrak{1}$ are trivial
\begin{align*}
\Omega_{\mathfrak{00}\ldots\mathfrak{0}} = \Omega_{\mathfrak{aa}\ldots\mathfrak{a}} = \Omega_{\mathfrak{11}\ldots\mathfrak{1}} = \mathcal{H}. 
\end{align*}

Another notable similarity of regions can be illustrated on the words $\mathfrak{01}$ and $\mathfrak{0011}$. 
Argumentation in~\cite[Section 4]{Hupkes2019b} shows that the region $\Omega_\mathfrak{0011}$ has exactly the same shape as twice vertically stretched region $\Omega_\mathfrak{01}$. 
Indeed, we can consider $\mathrm{u}_1 = \mathrm{u}_2$ and $\mathrm{u}_3 = \mathrm{u}_4$ for solution of type $\mathfrak{0011}$ and the system~\eqref{eq:GDEstat} reduces to two equations with halved diffusion coefficient $d$. 
We were however not able to generalize this observation to other types of solutions since, e.g., the natural candidate $\Omega_\mathfrak{000111}$ does not possess this property since $\mathrm{u}_1  \neq \mathrm{u}_2 \neq \mathrm{u}_3$ holds in general.  

Motivated by the previous paragraphs, we define the notion of similarity of the sets $\Omega_\mathrm{w}$. 

\begin{defin}
\label{def:regions}
Two regions $\Omega_{\mathrm{w}_1}, \Omega_{\mathrm{w}_2} \subset \mathcal{H}$ are called \textit{qualitatively equivalent} if either
\begin{align*}
\Omega_{\mathrm{w}_1} = \Omega_{\mathrm{w}_2} \quad \mathrm{or} \quad \Omega_{\mathrm{w}_1} = \mathcal{T}(\Omega_{\mathrm{w}_2}).
\end{align*} 
Two sets are called \textit{qualitatively distinct} if they are not qualitatively equivalent. 
\end{defin}

\subsection{Orbits and equivalence classes} 
Orbit of a word from $\mathcal{A}_3^n$ is a subset of $\mathcal{A}_3^n$ reachable by the action of some group $G$. 
As indicated in the previous section, we are interested in the number of different orbits since each orbit with respect to the group $D_n^\Pi$ contains words whose respective regions are qualitatively equivalent. 
In fact, the orbits divide the sets of words $\mathcal{A}_2^n, \mathcal{A}_3^n$ into equivalence classes, i.e., two words $\mathrm{w}_1, \mathrm{w}_2$ belong to the same equivalence class (have the same orbit) if there exists a group operation $g \in G$ such that $\mathrm{w}_1 = g(\mathrm{w}_2)$.
Burnside's lemma (Theorem~\ref{lem:burn}) and M\"{o}bius inversion formula (Theorem~\ref{thm:mobius}) are the main tools for determining the number of the classes and the classes representing words with a given primitive period, respectively. 

\begin{ex}
\label{ex:33}
There are $27$ words of length $n=3$ made with the alphabet $\mathcal{A}_3 = \{ \mathfrak{0}, \mathfrak{a}, \mathfrak{1} \}$:
\begin{align*}
\mathrm{W}_{\! \mathcal{A}_3}(3) = & \{ \mathfrak{000}, \mathfrak{00a}, \mathfrak{001}, \mathfrak{0a0}, \mathfrak{0aa}, \mathfrak{0a1}, \mathfrak{010}, \mathfrak{01a}, \mathfrak{011}, \\
&\mathfrak{a00}, \mathfrak{a0a}, \mathfrak{a01}, \mathfrak{aa0}, \mathfrak{aaa}, \mathfrak{aa1}, \mathfrak{a10}, \mathfrak{a1a}, \mathfrak{a11}, \\
&\mathfrak{100}, \mathfrak{10a}, \mathfrak{101}, \mathfrak{1a0}, \mathfrak{1aa}, \mathfrak{1a1}, \mathfrak{110}, \mathfrak{11a}, \mathfrak{111} \} .
\end{align*}
Taking into account the action of the group $C_3$, there are $11$ equivalence classes
\begin{align*}
\mathrm{W}_{\! \mathcal{A}_3}^{C_3}(3) = & 
\big\{
\{ \mathfrak{000} \}, \{ \mathfrak{aaa} \}, \{ \mathfrak{111} \}, \\
&
\{ \mathfrak{00a}, \mathfrak{0a0}, \mathfrak{a00} \}, 
\{ \mathfrak{001}, \mathfrak{010}, \mathfrak{100} \},
\{ \mathfrak{0aa}, \mathfrak{aa0}, \mathfrak{a0a} \}, \\
&
\{ \mathfrak{0a1}, \mathfrak{a10}, \mathfrak{10a} \},
\{ \mathfrak{01a}, \mathfrak{1a0}, \mathfrak{a01} \}, 
\{ \mathfrak{011}, \mathfrak{110}, \mathfrak{101} \}, \\
&
\{ \mathfrak{aa1}, \mathfrak{a1a}, \mathfrak{1aa} \}, 
\{ \mathfrak{a11}, \mathfrak{11a}, \mathfrak{1a1} \} \big\}, 
\end{align*}
while the action of the group $D_3$ merges two of these classes 
\begin{align*}
\mathrm{W}_{\! \mathcal{A}_3}^{D_3}(3) = & 
\big\{
\{ \mathfrak{000} \}, \{ \mathfrak{aaa} \}, \{ \mathfrak{111} \}, \\
&
\{ \mathfrak{00a}, \mathfrak{0a0}, \mathfrak{a00} \}, 
\{ \mathfrak{001}, \mathfrak{010}, \mathfrak{100} \},
\{ \mathfrak{0aa}, \mathfrak{aa0}, \mathfrak{a0a} \}, \\
&
\{ \mathfrak{0a1}, \mathfrak{a10}, \mathfrak{10a},
\mathfrak{1a0}, \mathfrak{a01}, \mathfrak{01a} \}, 
\{ \mathfrak{011}, \mathfrak{110}, \mathfrak{101} \}, \\
&
\{ \mathfrak{aa1}, \mathfrak{a1a}, \mathfrak{1aa} \}, 
\{ \mathfrak{a11}, \mathfrak{11a}, \mathfrak{1a1} \} \big\}.
\end{align*}

The action of the groups $C_3^\Pi$ and $D_3^\Pi$ divides the set of the words into the same system of $6$ equivalence classes
\begin{align*}
\mathrm{W}_{\! \mathcal{A}_3}^{C_3^\Pi}(3) =  \mathrm{W}_{\! \mathcal{A}_3}^{D_3^\Pi}(3) = & 
\big\{
\{ \mathfrak{000},  \mathfrak{111} \}, \{ \mathfrak{aaa} \}, \\
&
\{ \mathfrak{00a}, \mathfrak{0a0}, \mathfrak{a00},  
\mathfrak{11a}, \mathfrak{1a1}, \mathfrak{a11} \}, 
\{ \mathfrak{001}, \mathfrak{010}, \mathfrak{100}, \\
& \mathfrak{110}, \mathfrak{101}, \mathfrak{011} \},
\{ \mathfrak{0aa}, \mathfrak{aa0}, \mathfrak{a0a}, 
\mathfrak{1aa}, \mathfrak{aa1}, \mathfrak{a1a} \}, \\
&
\{ \mathfrak{0a1}, \mathfrak{a10}, \mathfrak{10a},
\mathfrak{1a0}, \mathfrak{a01}, \mathfrak{01a} \}
\big\}.
\end{align*}
See Figure~\ref{fig:diagram_3_n_3} for a graphical illustration of the equivalence classes. 
\begin{figure}[ht]

\centering
\includegraphics[width = \linewidth]{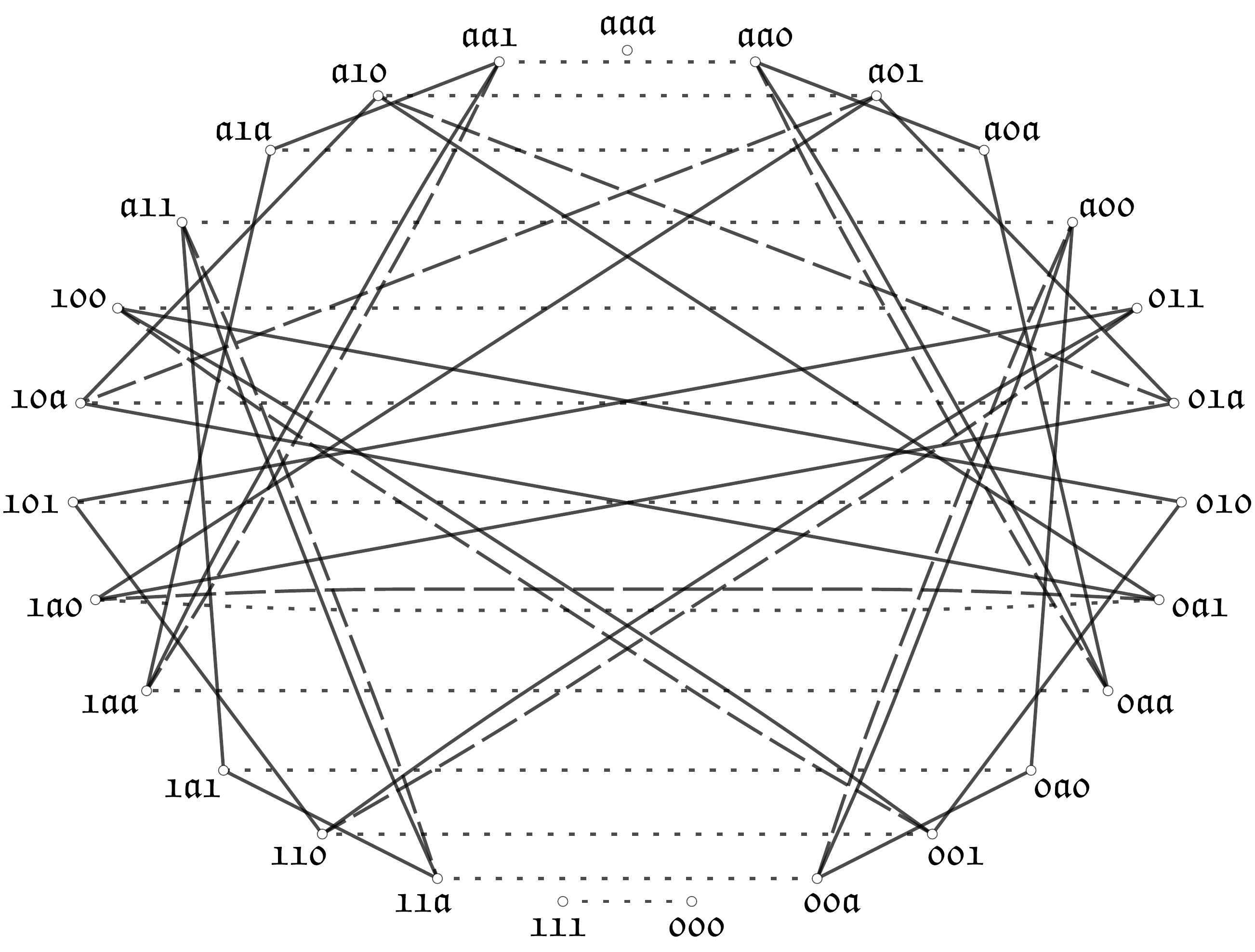}
\caption{A diagram capturing the action of the groups $C_3$, $D_3$, $C_3^\Pi$ and $D_3^\Pi$ on the set of three-letter words made with the alphabet $\mathcal{A}_3$. 
The presence of a line connecting two words indicates the existence of an operation transforming the solutions onto each other. 
The rotation $r$ is expressed by a solid line, the reflection $s$ is expressed by a dashed line and the symbol permutation $\pi$ is expressed by a dotted line. 
Every maximal connected subgraph with appropriate line types represents one equivalence class with respect to the action of a certain group, e.g., the action of $C_3^\Pi$ is depicted by solid and dotted lines. }
\label{fig:diagram_3_n_3}
\end{figure}

\end{ex}

The crucial question is whether we can determine the number of equivalence classes in a systematic manner. 
A useful tool for this is Burnside's lemma~\cite{Burnside1911}. 

\begin{thm}[Burnside's lemma]
\label{lem:burn}
Let $G$ be a finite group operating on a finite set $S$. 
Let $I(g)$ be the number of set elements such that the group operation $g \in G$ leaves them invariant. 
Then the number of distinct orbits $O$ is given by the formula
\[
O = \frac{1}{|G|} \sum_{g \in G} I(g). 
\]
\end{thm}
The power of Burnside's lemma lies in the fact that one counts fixed points of the group operations instead of the orbits themselves.
This can be much simpler in many cases as can be seen in the forthcoming sections.

The number of the orbits induced by the action of the group $C_n$ is usually called the number of the \textit{necklaces} made with $n$ beads in two (the alphabet $\mathcal{A}_2$) or three (the alphabet $\mathcal{A}_3$) colors. 
The \textit{bracelets} are induced by the action of the dihedral group $D_n$. 
Due to the lack of a common terminology, we call the classes induced by the action of the groups $C_n^\Pi$ and $D_n^\Pi$ the \textit{permuted necklaces} and the \textit{permuted bracelets}, respectively. 

Burnside's lemma does not take into account the primitive period of the words. 
For example, the existence region of the solutions of type $\mathfrak{0a1}$ coincides with the region of $\mathfrak{0a10a1}$ and thus cannot be counted twice. 
The assumption of the primitive period of a given length together with the action of the cyclic group $C_n$ create classes which are called the \textit{Lyndon necklaces}. The \textit{Lyndon bracelets} are a natural counterpart resulting from the action of the dihedral group $D_n$ together with the assumption of  a given primitive period length. 
The classes representing words with a given primitive period length without specification of the group are called the \textit{Lyndon words}. 
We emphasize that the terminology is not fully unified in the literature but the one presented here suits our purpose best without rising any unnecessary confusion. 

\begin{thm}[M\"{o}bius inversion formula]
\label{thm:mobius}
Let $f,g:\mathbb{N} \to \mathbb{R}$ be two arithmetic functions such that 
\begin{align*}
f(n) = \sum_{d|n} g(d) , 
\end{align*}
holds for all $n \in \mathbb{N}$. 
Then the values of the latter function $g$ can be expressed as
\begin{align*}
g(n) = \sum_{d|n} \mu \left(\frac{n}{d} \right) f(d),
\end{align*}
where $\mu$ is the M\"{o}bius function. 
\end{thm}

The M\"{o}bius function $\mu$ was first introduced in~\cite{Mobius1832} as
\begin{align*}
\mu(n) = 
\begin{cases}
(-1)^{P(n)}, & \text{each prime factor of $n$ is present at most once, }\\
0, & \text{otherwise,} \\
\end{cases}
\end{align*}
where $P(n)$ number of the prime factors of $n$. 
Use of M\"{o}bius inversion formula is a straightforward one. 
Let us assume, that we know the number $f(n)$ of the equivalence classes induced by the action of one of the above defined groups ($C_n$, $D_n$, $C_n^\Pi$, $D_n^\Pi$) for each $n \in \mathbb{N}$ (note that the group actions preserve the length of the primitive period of each of the words). 
Then for each $n$, this number $f(n)$ is given as the sum of the number of equivalence classes representing the words with primitive period of length $d$ dividing $n$. 

In the further text, we extensively exploit two crucial properties of M\"{o}bius inversion formula. 
Firstly, the formula is linear in the sense, that 
\begin{align*}
\sum_{i=1}^m \alpha_i f_i(n) = \sum_{d|n} g(d)
\end{align*}
implies 
\begin{align*}
g(n) = \sum_{d|n} \mu \left(\frac{n}{d} \right) \sum_{i=1}^m \alpha_i f_i(d) = \sum_{i=1}^m \alpha_i \sum_{d|n} \mu \left(\frac{n}{d} \right) f_i(d),
\end{align*}
and thus, each $f_i$ can be treated separately. 
Secondly, we can freely exchange indices in the following manner 
\begin{align*}
g(n) = \sum_{d|n} \mu \left(\frac{n}{d} \right) f(d) = \sum_{d|n} \mu(d) \, f\left(\frac{n}{d} \right) 
\end{align*}
since $n = n/d \cdot d$.

\begin{ex}
\label{ex:42}
We complement Example~\ref{ex:33} with the list of equivalence classes of the of the words with length $n=4$ made with the alphabet $\mathcal{A}_2 = \{ \mathfrak{0}, \mathfrak{1} \}$. 
There exist words of length $4$ with the primitive period $2$ and thus the set of the equivalence classes and the set of the Lyndon words will differ not by only the trivial constant words $\{ \mathfrak{0000}\}, \{ \mathfrak{1111} \}$. 
There are $16$ words of length $4$
\begin{align*}
\mathrm{W}_{\! \mathcal{A}_2}(4) = 
\big\{&
\mathfrak{0000}, \mathfrak{0001}, 
\mathfrak{0010}, \mathfrak{0011},
\mathfrak{0100}, \mathfrak{0101},
\mathfrak{0110}, \mathfrak{0111}, \\
&
\mathfrak{1000}, \mathfrak{1001}, 
\mathfrak{1010}, \mathfrak{1011},
\mathfrak{1100}, \mathfrak{1101},
\mathfrak{1110}, \mathfrak{1111}
\big\}.
\end{align*}
We next include the equivalence classes induced by the action of the groups $C_4$, $D_4$, $C_4^\Pi$, $D_4^\Pi$. 
The words with the primitive period of length smaller than $4$ are highlighted by a grey color
\begin{align*}
\mathrm{W}_{\! \mathcal{A}_2}^{C_4}(4) = 
\mathrm{W}_{\! \mathcal{A}_2}^{D_4}(4) = 
\big\{&
\{\textcolor{gray}{\mathfrak{0000}} \}, 
\{\textcolor{gray}{\mathfrak{1111}} \},
\{
\textcolor{gray}{\mathfrak{0101}, \mathfrak{1010}}
\},
\{ \mathfrak{0011,0110,1100,1001} \}, \\
& 
\{ \mathfrak{0001, 0010, 0100, 1000 }\},
\{ \mathfrak{0111,1110,1101,1011} \}
\big\}, \\
\mathrm{W}_{\! \mathcal{A}_2}^{C_4^\Pi}(4) = 
\mathrm{W}_{\! \mathcal{A}_2}^{D_4^\Pi}(4) = 
\big\{&
\{\textcolor{gray}{\mathfrak{0000, 1111}} \},
\{
\textcolor{gray}{\mathfrak{0101}, \mathfrak{1010}}
\},
\{ \mathfrak{0011,0110,1100,1001} \}, \\
& 
\{ \mathfrak{0001, 0010, 0100, 1000 },\mathfrak{1110, 1101, 1011, 0111} \}
\big\}.
\end{align*}
\end{ex}

This introduction allows us to state the main theorem of the paper which gives an upper estimate of qualitatively distinct regions belonging to words of length $m$ which ranges from one up to some given value $n \in \mathbb{N}$. 
We must combine the action of the dihedral group $D_m^\Pi$ with the assumption of the primitive period equal to the word length (Lyndon bracelets) for each $m \leq n$. 
It is however upper estimate only, since we cannot be sure whether there exist two qualitatively equivalent regions whose labelling words are not related via any of the above mentioned symmetries. 
Numerical simulations however indicate that the upper bound may be close to optimal,~\cite{Hupkes2019b}. 

Since the expressions in the theorem may look confusing at the first sight we include a short preliminary commentary. 
The function $B\!L_k^\pi(m)$ denotes the number of the permuted Lyndon bracelets of length $m$ and the formulas are defined by parts since they incorporate the number of the bracelets which cannot be written in a consistent form for even and odd $m$'s. 
The functions $\#_{\mathcal{A}_k}^{\leq}(n)$ just add the numbers of Lyndon bracelets of length ranging from two to $n$ including the one region $\Omega_\mathfrak{0} = \Omega_\mathfrak{a}=\Omega_\mathfrak{1}=\mathcal{H}$ identical for all homogeneous solutions. 

\begin{thm}
\label{thm:main}
Let $n \geq 2$ be given. 
There are at most 
\begin{align}
\label{eq:total3_intro}
\#_{\mathcal{A}_3}^{\leq}(n) = 1 + \sum_{m=2}^{n} B\!L_{\mathcal{A}_3}^\pi(m)
\end{align}
qualitatively distinct regions $\Omega_\mathrm{w}$, $\mathrm{w} \in \mathcal{A}_3^m$, out of which at most 
\begin{align}
\label{eq:total2_intro}
\#_{\mathcal{A}_2}^{\leq}(n) = 1 + \sum_{m=2}^{n} B\!L_{\mathcal{A}_2}^\pi(m)
\end{align}
regions belong to the asymptotically stable stationary solutions, where 
\begin{align}
\label{eq:brac3LynAper_main}
B\!L_{\mathcal{A}_3}^\pi(m) &= \frac{1}{4m}\bigg[ \sum_{d|m, \, d \, \text{odd}} \mu (d) \, 3^\frac{m}{d} + X_{N\!L}(m) + 2m \sum_{d|m} \mu\left( \frac{m}{d}\right) X^\pi_{B,3}(d)\bigg], \\
\label{eq:brac2LynAper_main}
B\!L_{\mathcal{A}_2}^\pi(m) &=
\frac{1}{4m}\bigg[ \sum_{d|m, \, d \, \text{odd}} \mu ( d ) \, 2^\frac{m}{d}   + 2n \sum_{d|m} \mu\left( \frac{m}{d}\right) X^\pi_{B,2}(d)\bigg], \\
X_{N\!L}(m) &= 
\begin{cases}
1, & m=1, \\ 
-1, & m=2^\alpha, \, \alpha \in \mathbb{N}, \\
0, & \text{otherwise}
\end{cases} \nonumber
\end{align} and 
\begin{align}
\label{eq:brac-auxil-intro}
X^\pi_{B,3}(d) &= 
\begin{cases}
\dfrac{4}{3} \cdot 3^{\frac{d}{2}} , & d \text{ is even}, \\
\\
2 \cdot 3^\frac{d-1}{2} , & d \text{ is odd},
\end{cases}  \quad
X^\pi_{B,2}(d) = 
\begin{cases}
2^{\frac{d}{2}} , & d \text{ is even}, \\
\\
2^\frac{d-1}{2} , & d \text{ is odd}.
\end{cases} 
\end{align}
\end{thm}

The formulas from Theorem~\ref{thm:main} are enumerated in Table~\ref{tab:vals}. 

\begin{table}[ht]
\centering
\begin{tabular}{c|cc|c@{\;}c@{\;}>{\raggedright}p{1.8cm}@{\;}c@{\;}c@{\;}p{2.2cm}<{\raggedright}}
$n$ & $3^n$ & $2^n$ & \begin{small}$\#_{\mathcal{A}_3}^{\leq}(n)$\end{small} & \begin{small}$B\!L_{\mathcal{A}_3}^\pi(n)$\end{small} &  & \begin{small}$\#_{\mathcal{A}_2}^{\leq}(n)$\end{small} & \begin{small}$B\!L_{\mathcal{A}_2}^\pi(n)$\end{small} & \\
[.2em]
\hline
\hline
2 & 9 & 4 & 3 & 2 & \begin{small} \{$\mathfrak{01}$, $\mathfrak{0a}$\} \end{small} & 2 & 1& \begin{small}\{$\mathfrak{01}$\} \end{small}\\
\hline
3 & 27 	& 8 	& 7 & 4 & \begin{small}\{$\mathfrak{00a}$, $\mathfrak{001}$, $\mathfrak{0a1}$, $\mathfrak{0aa}$\}  \end{small} & 3 & 1 & \begin{small}$\{ \mathfrak{001} \}$ \end{small}   
 	\\
 \hline
 4 & 81 	& 16 	& 16 & 9& 
 \begin{small}\{$\mathfrak{000a}$, $\mathfrak{0001}$, $\mathfrak{00aa}$, $\mathfrak{00a1}$, $\mathfrak{0011}$, $\mathfrak{0a0a}$, $\mathfrak{0aaa}$, $\mathfrak{0aa1}$, $\mathfrak{0a1a}$\}\end{small} 	& 5 &  2 & \begin{small}$\{ \mathfrak{0001,0011} \}$ \end{small}	\\
  \hline
 5 & 243 	& 32 	& 36 & 20 & \multicolumn{1}{c}{\begin{small}\textcolor{gray}{not listed}\end{small}}	& 8 & 3 & \begin{small}\{$\mathfrak{00001}$, $\mathfrak{00011}$, $\mathfrak{00101}$ \} \end{small} 	\\
  \hline
 6 & 729 	& 64 	& 80 & 44 & \multicolumn{1}{c}{\begin{small}\textcolor{gray}{not listed}\end{small}} & 13  & 5 & \begin{small}$\{ \mathfrak{000001}$, $\mathfrak{000011}$, $\mathfrak{000101}$, $\mathfrak{000111}$, $\mathfrak{001011}\}$\end{small}	\\
  \hline
 7 & 2187 	& 128 	& 184 & 104 & \multicolumn{1}{c}{\begin{small}\textcolor{gray}{not listed}\end{small}} & 21 & 8 & \multicolumn{1}{c}{\begin{small}\textcolor{gray}{not listed}\end{small}} \\
\hline
 8 & 6561 	& 256 	& 437 & 253 & \multicolumn{1}{c}{\begin{small}\textcolor{gray}{not listed}\end{small}}	& 35 & 14 & \multicolumn{1}{c}{\begin{small}\textcolor{gray}{not listed}\end{small}}	\\
  \hline
 9 & 19683 	& 512	& 1061 & 624 & \multicolumn{1}{c}{\begin{small}\textcolor{gray}{not listed}\end{small}}	& 56 & 21 & \multicolumn{1}{c}{\begin{small}\textcolor{gray}{not listed}\end{small}}	\\
  \hline
 10 & 59049 & 1024 	& 2689 & 1628 & \multicolumn{1}{c}{\begin{small}\textcolor{gray}{not listed}\end{small}}	& 95 & 39 & \multicolumn{1}{c}{\begin{small}\textcolor{gray}{not listed}\end{small}}
\end{tabular}
\caption{Enumerated formulas from Theorem~\ref{thm:main}. The columns for $3^n$ and $2^n$ are added for comparison since there are in total $3^n$ regions $\Omega_\mathrm{w}$ with $\mathrm{w} \in \mathcal{A}_3^n$ and $2^n$ of them correspond to the asymptotically stable stationary solutions.
The unlabelled columns list the lexicographically smallest representatives of the Lyndon bracelets of a given length created with the respective alphabets; further lists are omitted to prevent clutter. 
Note that $\#_{\mathcal{A}_k}^{\leq}(n+1) = \#_{\mathcal{A}_k}^{\leq}(n) + B\!L_{\mathcal{A}_k}^\pi(n+1)$ holds for $n \geq  2$ and $k =2,3$. 
} 
\label{tab:vals}
\end{table}

\subsection{Known results}
Here, we summarize known results relevant to the focus of this paper. 
This summary consists of two parts since our main result, Theorem~\ref{thm:main}, contributes to the knowledge of the periodic stationary solutions of the LDE~\eqref{eq:LDE} as well as to the theory of combinatorial enumeration. 

The number of equivalence classes with respect to the action of the groups $C_n$ and $D_n$ and their connection to the stationary solutions of the GDE~\eqref{eq:GDE} and the LDE~\eqref{eq:LDE} were studied in the paper~\cite{Hupkes2019c}. 
The results considered all stationary solutions (words from $\mathcal{A}_3^n$) as well as the stable solutions (words from $\mathcal{A}_2^n$). 
M\"{o}bius inversion formula was used therein to determine the numbers of the Lyndon necklaces and the Lyndon bracelets. 
 
A more general case of the group $C_n^\Pi$ which acted on the set of words created with a given number of symbols not necessarily less or equal to three was considered in~\cite{Fine1958}. 
The author also simplified the counting formulas for the permuted necklaces and the permuted Lyndon necklaces for the case of two symbols, i.e., the alphabet $\mathcal{A}_2$, to the form which also appears in this paper, Lemmas~\ref{lem:neck2Aper} and~\ref{lem:neck2LynAper}. 
However, none of the presented results could be directly applied to the case of the transformation $\pi$ acting on the words from $\mathcal{A}_3^n$. 
Formally, the studied object was the group product of a cyclic group $C_n$ and a symmetric group $S_k$ (the group of all permutations of $k$ symbols). 
This coincides with our case only if $k=2$, i.e., the words are created with a two symbol alphabet $\mathcal{A}_2$. 
If $k=3$, then the group $C_n^\Pi$ is isomorphic to the group product $C_n \times G$ where $G$ is only a specific subgroup of $S_3$.   
Let us also mention that the problem was studied from the combinatorial point of view. 

The authors in~\cite{Gilbert1961} were among other results able to derive a general counting formula for the permuted bracelets and the permuted Lyndon bracelets of words created with an arbitrary number of symbols. 
As in the case of the necklaces in~\cite{Fine1958}, the results  relevant to this paper cover the case of the reduced alphabet $\mathcal{A}_2$ only. 
The generality of the presented formulas however comes with a cost of their complexity. 
Taking advantage of our more specific setting, we are able to utilize alternative approach which enables us to further simplify the formulas for the case of the words from $\mathcal{A}_2^n$. 
Also, the focus of the work lied mainly in clarifying certain combinatorial concepts. 

\section{Counting of equivalence classes}
We continue with listing and deriving auxiliary counting formulas as well as those which are directly used to prove the main result, Theorem~\ref{thm:main}.

In this section, $(m,n)$ denotes the greatest common divisor of $m,n \in \mathbb{N}$. 

\subsection{Counting of non-Lyndon words}
We start with counting of the necklaces of length $n$ made with $k$ symbols. 
\begin{lem}[{\cite[p.~162]{Riordan1958}}]
\label{lem:neck}
Given $n\in \mathbb{N}$, the number of equivalence classes induced by the action of the group $C_n$ on the set of all words of length $n$ made with a $k$-symbol alphabet is 
\begin{align}
\label{eq:neck}
N_k(n) = \frac{1}{n} \sum_{d|n} \varphi(d) k^{\frac{n}{d}} . 
\end{align}
\end{lem}
 
The function $\varphi(d)$ is the Euler totient function which counts relatively coprime numbers to $d$, see~\cite{Apostol1976}.
Another classical result concerns the number of the bracelets of length $n$ made with $k$ symbols. 

\begin{lem}[{\cite[p.~150]{Riordan1958}}]
\label{lem:brac}
Given $n\in \mathbb{N}$, the number of equivalence classes induced by the action of the group $D_n$ on the set of all words of length $n$ made with a $k$-symbol alphabet is 
\begin{align}
\label{eq:brac}
B_k(n) = 
\dfrac{1}{2}\big[ N_k(n)+ X_{B,k}(n) \big],
\end{align}
where 
\begin{align}
\label{eq:brac_auxil}
X_{B,k}(n) = 
\begin{cases}
\dfrac{k+1}{2} \, k^\frac{n}{2}, & n \text{ is even}, \\
\\
k^\frac{n+1}{2} , & n \text{ is odd}.
\end{cases}
\end{align}
\end{lem}

The formulas for the necklaces and the bracelets can be derived for a general number of symbols $k$. 
If we take the symbol permutation $\pi$ into the account, the formulas regarding the alphabets $\mathcal{A}_2$ and $\mathcal{A}_3$ are slightly different and thus, we treat both cases separately. 
The main difference is that there are no invariant words with respect to the value permutation $\pi$ with the alphabet $\mathcal{A}_2$ and $n$ odd. 
Indeed, the necessary condition for the invariance is that the word has the same number of $\mathfrak{0}$'s and $\mathfrak{1}$'s. 
This can be bypassed by the use of the symbol $\mathfrak{a}$ from the alphabet $\mathcal{A}_3$. 

\begin{lem}[{\cite[p.~300]{Fine1958}}]
\label{lem:neck2Aper}
Given $n \in \mathbb{N}$, the number of equivalence classes induced by the action of the group $C_n^\Pi$ on the set of all words of length $n$ made with the alphabet $\mathcal{A}_2$ is
\begin{align}
\label{eq:neck2Aper}
N^\pi_{\mathcal{A}_2}(n) = 
\frac{1}{2n} \bigg[ \sum_{d|n, \, d \, \text{odd}} \varphi(d) \, 2^{\frac{n}{d}} + 2 \sum_{d|n, \, d \, \text{even}} \varphi(d) \, 2^\frac{n}{d} \bigg]. 
\end{align}
\end{lem}

A somewhat similar formula can be derived for the necklaces made with the three-letter alphabet $\mathcal{A}_3$. 

\begin{lem}
\label{lem:neck3Aper}
Given $n \in \mathbb{N}$, the number of equivalence classes induced by the action of the group $C_n^\Pi$ on the set of all words of length $n$ made with the alphabet $\mathcal{A}_3$ is 
\begin{align}
\label{eq:neck3Aper}
N^\pi_{\mathcal{A}_3}(n) = 
\frac{1}{2n} \bigg[ \sum_{d|n, \, d \, \text{odd}} \varphi(d) \left(1+3^{\frac{n}{d}} \right) + 2 \sum_{d|n, \, d \, \text{even}} \varphi(d) \, 3^\frac{n}{d} \bigg]. 
\end{align}
\end{lem}
\begin{proof}
The group $C_n^\Pi$ contains the pure rotations $r_i$ and the rotations with the symbol permutations $r\pi_i$ totalling $2n$ operations. 
A direct application of Burnside's lemma (Theorem~\ref{lem:burn}) yields
\begin{align*}
N^\pi_{\mathcal{A}_3}(n) &= \frac{1}{2n} \bigg[\sum_{l=0}^{n-1} I(r^l) + \sum_{l=0}^{n-1} I(\pi r^l) \bigg]. 
\end{align*}
The expression~\eqref{eq:neck} in the context of Lemma~\ref{lem:neck} shows that 
\begin{align*}
\sum_{l=0}^{n-1} I(r^l) = \sum_{d|n} \varphi(d) 3^{\frac{n}{d}}. 
\end{align*}

Given $l =0,1, \ldots, n-1$, the aim is to express the general form of a word $\mathrm{w}$ invariant to the operation $\pi r^l$. 
A rotation by $l$ positions induces a permutation of the word's $\mathrm{w}$ letters with $(n,l)$ cycles of length $n/(n,l)$.
The word $\mathrm{w}$ is then divided into $n/(n,l)$ disjoint subwords of length $(n,l)$.   
Assume that the first $(n,l)$ letters of the word $\mathrm{w}$ are given. 
A repeated application of the operation $\pi r^l$ then determines the form of all the remaining subwords of length $(n,l)$. 
Indeed, the rotation by $l$ positions applied to a word of length $n$ induces a rotation by $l/(n,l)$ positions of the $n/(n,l)$ subwords because $l/(n,l)$ and $n/(n,l)$ are relatively coprime. 
Here, the parity of the subwords' number $n/(n,l)$ must be considered.  
If $n/(n,l)$ is odd, then the only possible word invariant to $\pi r^l$ is constant $\mathfrak{a}$'s.  
The even $n/(n,l)$ allows $3^{(n,l)}$ possible words.

Let us pick an arbitrary divisor $d$ of $n$. Then, surely $d=n/(n,l)$ for some $l \in \{0,1, \ldots, n-1\}$. 
The cyclic group $C_d$ with $d$ elements can be generated by $\varphi(d)$ different values relatively coprime to $d$.  

The argumentation above results in
\begin{align*}
N^\pi_{\mathcal{A}_3}(n) &= \frac{1}{2n} \bigg[\sum_{l=0}^{n-1} I(r^l) + \sum_{l=0}^{n-1} I(\pi r^l) \bigg], \\
&= \frac{1}{2n} \bigg[ \sum_{d|n} \varphi(d) \, 3^{\frac{n}{d}}  + \sum_{d|n, \, d \, \text{odd}} \varphi(d) + \sum_{d|n, \, d \, \text{even}} \varphi(d) \, 3^{\frac{n}{d}} \bigg],  \\
&= \frac{1}{2n} \bigg[ \sum_{d|n, \, d \, \text{odd}} \varphi(d) \left(1+3^{\frac{n}{d}} \right) + 2 \sum_{d|n, \, d \, \text{even}} \varphi(d) \, 3^\frac{n}{d} \bigg].  \qedhere
\end{align*}

\end{proof}

We now approach to the formulas regarding the group $D_n^\Pi$; the permuted brace\-lets. 
As in the previous text, we treat the cases of the alphabets $\mathcal{A}_2$ and $\mathcal{A}_3$ separately. 
A general counting formula regarding the alphabet $\mathcal{A}_2$ as a special case was derived in~\cite{Gilbert1961}. 
We present an alternative proof which can be generalized to the case of the alphabet $\mathcal{A}_3$. 

\begin{lem}
\label{lem:brac2Aper}
Given $n \in \mathbb{N}$, the number of equivalence classes induced by the action of the group $D_n^\Pi$ on the set of all words of length $n$ made with the alphabet $\mathcal{A}_2$ is 
\begin{align}
\label{eq:brac2Aper}
B_{\mathcal{A}_2}^\pi(n) = 
\dfrac{1}{2} \big[ N_{\mathcal{A}_2}^\pi(n) + X^\pi_{B,2}(n) \big],
\end{align}
where 
\begin{align}
\label{eq:brac2_auxil}
X^\pi_{B,2}(n) = 
\begin{cases}
2^{\frac{n}{2}} , & n \text{ is even}, \\
\\
2^\frac{n-1}{2} , & n \text{ is odd}.
\end{cases}
\end{align}
\end{lem}

\begin{proof}
The group $D_n^\Pi$ contains the rotations $r^i$, the rotations with the reflection $sr^i$, the rotations with the symbol permutation $\pi r^i$ and the rotations with the reflection and the symbol permutation $\pi s r^i$. 
Burnside's lemma (Theorem~\ref{lem:burn}) then yields
\begin{align*}
N^\pi_{\mathcal{A}_3}(n) &= 
\frac{1}{4n} \bigg[\sum_{l=0}^{n-1} I(r^l) + \sum_{l=0}^{n-1} I(\pi r^l) + \sum_{l=0}^{n-1} I(sr^l) + \sum_{l=0}^{n-1} I(\pi s r^l) \bigg]. 
\end{align*}

The equivalence classes induced by the transformations $r^l$ and $\pi r^l$ are enumerated in the expression~\eqref{eq:neck2Aper} of Lemma~\ref{lem:neck2Aper}. 
Each line in formula~\eqref{eq:brac_auxil} counts the number of orbits with respect to the rotation with reflection $s r^l$. 

First, we clarify certain concepts valid for the operation $sr^l$ and subsequently apply them to the case of $\pi sr^l$. 
The composition of the rotation and the reflection is not commutative in general, but $sr^l =r^l \, \circ \, sr^0 = sr^0 \, \circ \, r^{n-l}$ holds for $l = 0,1, \ldots, n-1$. 
This formula and the group associativity yields
\begin{align*}
sr^l \, \circ \, sr^l = (r^l \, \circ \, sr^0) \circ ( sr^0 \,\circ \, r^{n-l}) = r^l \, \circ \, r^{n-l} = r^0 = e.
\end{align*}
Thus, the induced permutation of the word's letters has cycles of the length 1 or 2 only. 

Given $l = 0, 1, \ldots, n-1$, then
\begin{align*}
\big(sr^l(\mathrm{w})\big)_i = \mathrm{w}_{n-l-i+1}
\end{align*}
for $i \leq \lceil l/2 \rceil$. 
Due to the composition formula, the positions from $n-l+1$ to $n$ transform accordingly. 
This induces a partition of the word $\mathrm{w}$ into two subwords, see Figure~\ref{fig:brac} for illustration.
\begin{figure}[ht]
\centering
\includegraphics[width =\linewidth]{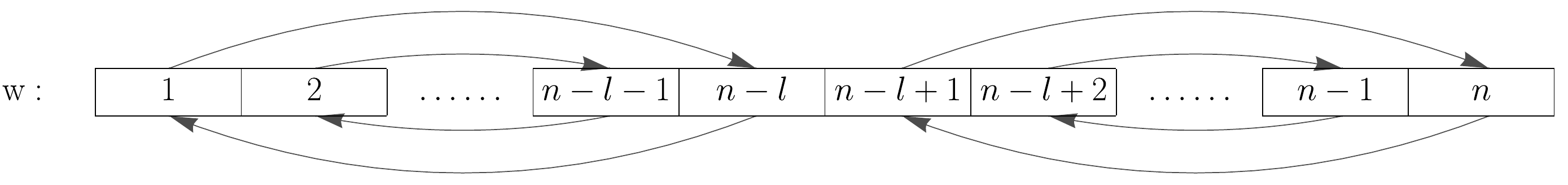}
\caption{
Illustration of operation of the group transformation $rs_l$ on the word $\mathrm{w}$ of length $n$. 
The transformation $sr^l$ divides the word $\mathrm{w}$ into two subwords whose elements starting from the edges map to each other. 
}
\label{fig:brac}
\end{figure}
The combined parities of $n$ and $l$ determine the parity of the subwords' length and thus whether there is a middle letter mapped to itself. 
For any subword of odd length, there is exactly one loop. 
All possible combinations are
\begin{center}
\begin{tabular}{c|cc}
$n \backslash l$ & even & odd \\
\hline
even & even, even & odd, odd \\
odd & odd, even & even, odd
\end{tabular}. 
\end{center} 

Let us now assume the operation $\pi s r^l$. 
If $n$ is odd, then one of the subwords induced by the action of $s r^l$ is always odd and thus there are no words fixed by $\pi sr^l$. 
If $n$ is even the only possibility for the word $\mathrm{w}$ to be fixed is when $l$ is also even. 
There are then $n/2$ cycles of length $2$ leading to $n/2 \, \cdot \, 2^{n/2}$ words fixed by the operation of the form $\pi s r^i$.  

The summing of all cases and including~\eqref{eq:brac_auxil} for $I(sr^l)$ results in
\begin{align*}
B_{\mathcal{A}_2}^\pi(n) \biggr\rvert_{n \text{ even}} &= 
\frac{1}{4n} \bigg[2n \cdot N_{\mathcal{A}_2}^\pi(n) + \frac{3n}{2} \cdot 2^\frac{n}{2} + \frac{n}{2} \cdot 2^\frac{n}{2}\bigg], \\
&= \frac{1}{2} \Big[ N_{\mathcal{A}_2}^\pi(n) + 2^\frac{n}{2}\Big], \\
 B_{\mathcal{A}_2}^\pi(n) \biggr\rvert_{n \text{ odd}} &= 
\frac{1}{4n} \Big[2n \cdot N_{\mathcal{A}_2}^\pi(n) + n \cdot 2^\frac{n+1}{2}\Big] \\
&= \frac{1}{2} \Big[ N_{\mathcal{A}_2}^\pi(n) + 2^\frac{n-1}{2}\Big]. \qedhere
\end{align*} 
\end{proof}

A general idea presented in the proof of Lemma~\ref{lem:brac2Aper} can be applied to the case of the three letter alphabet $\mathcal{A}_3$. 

\begin{lem}
\label{lem:brac3Aper}
Given $n \in \mathbb{N}$, the number of equivalence classes induced by the action of the group $D_n^\Pi$ on the set of all words of length $n$ made with the alphabet $\mathcal{A}_3$ is 
\begin{align}
\label{eq:brac3Aper}
B_{\mathcal{A}_3}^\pi(n) = 
\dfrac{1}{2} \big[ N_{\mathcal{A}_3}^\pi(n) + X^\pi_{B,3}(n) \big],
\end{align}
where
\begin{align}
\label{eq:brac3_auxil}
X^\pi_{B,3}(n) = 
\begin{cases}
\dfrac{4}{3} \, \cdot \, 3^{\frac{n}{2}} , & n \text{ is even}, \\
\\
2 \, \cdot \, 3^\frac{n-1}{2} , & n \text{ is odd}.
\end{cases}
\end{align}
\end{lem}

\begin{proof}
As in the proof of Lemma~\ref{lem:brac2Aper}, the only operations to be considered in detail are of the form $\pi sr^i$. 
For the sake of completeness, we note that there are $2n \cdot 3^{n/2}$ (for $n$ even) and $n \cdot 3^{(n+1)/2}$ (for n odd) words invariant to the action of transformations of the form $sr^i$, see~\eqref{eq:brac_auxil}. 

Let $l=0,1, \ldots, n-1$ be given. 
The operation $sr^l$ induces a letter permutation with cycles of length $1$ or $2$. 
In order for the word $\mathrm{w}$ to be fixed by the operation $\pi s r^l$, positions in the cycle of length $1$ can contain the letter $\mathfrak{a}$ only.

If $n$ is odd, then there are $(n-1)/2$ cycles of length $2$ leading to $n \cdot 3^{(n-1)/2}$ fixed words. 
If $n$ is even, then there are two cycles of length $1$ only if $l$ is odd. 
Summing over all $l = 0, \ldots, n-1$ leads to $n/2 \cdot (3^{n/2-1}+ 3^{n/2})$. 

The summary of the results gives
\begin{align*}
B_{\mathcal{A}_3}^\pi(n) \biggr\rvert_{n \text{ even}} &= 
\frac{1}{4n} \bigg[2n \cdot N_{\mathcal{A}_3}^\pi(n) + 2n \cdot 3^\frac{n}{2} + \frac{n}{2} \cdot (3^{\frac{n}{2}-1}+ 3^{\frac{n}{2}})\bigg], \\
&= \frac{1}{2} \bigg[ N_{\mathcal{A}_3}^\pi(n) + \frac{4}{3} \cdot 3^\frac{n}{2}\bigg], \\
 B_{\mathcal{A}_3}^\pi(n) \biggr\rvert_{n \text{ odd}} &= 
\frac{1}{4n} \Big[2n \cdot  N_{\mathcal{A}_2}^\pi(n) + n \cdot 3^\frac{n+1}{2}+n \cdot 3^\frac{n-1}{2}\Big], \\
&= \frac{1}{2} \Big[ N_{\mathcal{A}_3}^\pi(n) + 2 \cdot 3^\frac{n-1}{2}\Big]. \qedhere
\end{align*} 
\end{proof}

\subsection{Counting of the Lyndon words}
To derive the forthcoming formulas, we use a special property of the M\"obius function $\mu$ and the Euler totient function $\varphi$; these functions are multiplicative. 
An arithmetic function $\psi: \mathbb{N} \to \mathbb{R}$ is multiplicative if and only if $\psi(1) = 1$ and $\psi(ab) = \psi(a) \psi(b)$ provided $a$ and $b$ are relatively coprime. 
To prove that two multiplicative functions $\psi_1, \psi_2$ are equal it is enough to show that $\psi_1(p^\alpha) = \psi_2(p^\alpha)$ for all prime $p$ and $\alpha \in \mathbb{N}$. 
For further information about the multiplicative functions see, e.g.,~\cite{Apostol1976}. 

We start with a technical lemma which is used later. 
\begin{lem}
The identity
\begin{align}
\label{eq:convAll}
\sum_{d|n} \mu\left(\dfrac{n}{d}\right)\dfrac{n}{d} \, \varphi(d) &= \mu(n)  
\end{align}
holds for any $n \in \mathbb{N}$.
Furthermore, the following identities hold for any $n$ even,
\begin{align}
\label{eq:convEven}
\sum_{d|n, \, d \, \text{even}} \mu\left(\dfrac{n}{d}\right)\dfrac{n}{d} \, \varphi(d) \;\; \biggr\rvert_{n \text{ even}} &= -\mu(n), \\
\label{eq:convOdd}
\sum_{d|n, \, d \, \text{odd}} \mu\left(\dfrac{n}{d}\right)\dfrac{n}{d} \, \varphi(d) \;\; \biggr\rvert_{n \text{ even}} &= 2 \mu(n). 
\end{align}
\end{lem}
\begin{proof}
The expression~\eqref{eq:convAll} is an equality of two multiplicative functions. 
It is sufficient to verify the formula for $n=p^\alpha$, where $p$ is a prime and $\alpha \in \mathbb{N}$,~\cite{Apostol1976}.  
If $\alpha \geq 2$ then $\mu(p^\alpha)=0$ and 
\begin{multline*}
\sum_{d|n} \mu\left(\dfrac{n}{d}\right)\dfrac{n}{d} \, \varphi(d) = \mu(p) \, p \, \varphi(p^{\alpha-1}) + \mu(1) \, \varphi(p^\alpha) =\\= -p^{\alpha-1} (p-1)+p^{\alpha-1} (p-1) = 0, 
\end{multline*}
if $\alpha = 1$, then $\mu(p^\alpha)= -1$ and 
\begin{align*}
\sum_{d|n} \mu\left(\dfrac{n}{d}\right)\dfrac{n}{d} \, \varphi(d) = \mu(p) \, p \, \varphi(1) + \mu(1) \, \varphi(p) = -p+ p-1 = -1, 
\end{align*}
if $\alpha =0$, then $\mu(p^\alpha)=1$ and
\begin{align*}
\sum_{d|n} \mu\left(\dfrac{n}{d}\right)\dfrac{n}{d} \, \varphi(d) = \mu(1) \, 1 \, \varphi(1) = 1. 
\end{align*}
This proves~\eqref{eq:convAll}. 

Let us assume, that the even integer $n \in \mathbb{N}$ has the form $n = 2^\beta P$, where $P$ is a product of odd primes. 
We can now rewrite~\eqref{eq:convOdd} as 
\begin{align*}
\sum_{d|n, \, d \, \text{odd}} \mu\left(\dfrac{n}{d}\right)\dfrac{n}{d} \, \varphi(d) = \sum_{d|n/2^\beta} \mu\left(\dfrac{n}{d}\right)\dfrac{n}{d} \, \varphi(d). 
\end{align*}
If $\beta \geq 1$, then the fraction $n/d$ always contains a squared prime factor and thus $\mu(n/d)=0$ which corresponds to $\mu(2^\beta P) = 0$. 
Suppose $\beta = 1$. 
We can now use the substitution $m = n/2$ together with the formula~\eqref{eq:convAll}
\begin{align*}
\sum_{d|n/2} \mu\left(\dfrac{n}{d}\right)\dfrac{n}{d} \, \varphi(d) &= \sum_{d|m} \mu\left(\dfrac{2m}{d}\right)\dfrac{2m}{d} \, \varphi(d), \\ 
&= -2 \sum_{d|m} \mu\left(\dfrac{m}{d}\right)\dfrac{m}{d} \,
 \varphi(d), \\
 &= -2 \mu(m) = -2 \mu\left(\dfrac{n}{2}\right) = 2\mu(n) .
\end{align*}
The first sign change is possible due to the fact that the fraction $m/d$ is an odd integer and thus $2$ is not part of its prime factorization. 
The second one utilizes the same idea. 
This concludes the proof of~\eqref{eq:convOdd}.

The identity~\eqref{eq:convEven} follows from~\eqref{eq:convAll} and~\eqref{eq:convOdd} since 
\begin{align*}
\sum_{d|n} f(d) = \sum_{d|n, \, d \, \text{even}} f(d) + \sum_{d|n, \, d \, \text{odd}} f(d),
\end{align*}
holds for any $n \in \mathbb{N}$. \qedhere
\end{proof}

The counting formula for the Lyndon necklaces can be derived by a direct argument as in~\cite{Golomb1958} but we choose more technical approach whose idea is useful in later proofs. 

\begin{lem}
\label{lem:neckLyn}
Given $n \in \mathbb{N}$, the number of the Lyndon necklaces (the group $C_n$) with period $n$ on the set of all words of length $n$ made with $k$ symbols is
\begin{align}
\label{eq:neckLyn}
N\!L_k(n) = \frac{1}{n} \sum_{d|n} \mu\left( \frac{n}{d}\right) \, k^d. 
\end{align}
\end{lem}

\begin{proof}
Since 
\begin{align*}
L_k(n) = \sum_{d|n} N\!L_k(n)
\end{align*}
holds for all $n \in \mathbb{N}$ the use of the M\"obius inversion formula (Theorem~\ref{thm:mobius}) and the subsequent substitution $d= ml$ yields
\begin{align*}
N\!L_k(n) &= \sum_{m|n} \mu \left( m \right) L_k \bigg( \frac{n}{m} \bigg), \\ 
& = \sum_{m|n} \mu \left( m \right) \frac{m}{n} \sum_{l| m/n} \varphi(l) \, k^\frac{n}{ml} , \\
& = \frac{1}{n} \sum_{d|n} k^\frac{n}{d} \sum_{l|d} \mu \bigg( \frac{d}{l} \bigg) \frac{d}{l} \, \varphi(l), \\
&= \frac{1}{n} \sum_{d|n} k^\frac{n}{d} \, \mu ( d ).  
\end{align*}
The last step uses~\eqref{eq:convAll}. \qedhere
\end{proof}

The counting formula for the Lyndon bracelets is a direct consequence of M\"obius inversion formula (Theorem~\ref{thm:mobius}) and Lemmas~\ref{lem:brac} and~\ref{lem:neckLyn}. 

\begin{lem}
\label{lem:bracLyn}
Given $n \in \mathbb{N}$, the number of the Lyndon bracelets (the group $D_n$) with period $n$ on the set of all words of length $n$ made with $k$ symbols is
\begin{align}
\label{eq:bracLyn}
B\!L_k(n) = \frac{1}{2} \bigg[N\!L_k(n) + \sum_{d|n} \mu\left( \frac{n}{d}\right) \, X_{B,k}(d) \bigg],  
\end{align}
where $X_{B,k}(d)$ is given by~\eqref{eq:brac_auxil}. 
\end{lem}

We continue with the counting formulas for the permuted Lyndon necklaces. 

\begin{lem}[{\cite[p.~301]{Fine1958}}]
\label{lem:neck2LynAper}
Let $n \in \mathbb{N}$ be given. The number of the permuted Lyndon necklaces (the group $C_n^\Pi$) with period $n$ on the set of all words of length $n$ made with the alphabet $\mathcal{A}_2$ is
\begin{align}
\label{eq:neck2LynAper}
N\!L_{\mathcal{A}_2}^\pi(n) = \frac{1}{2n} \sum_{d|n, \, d \, \text{odd}} \mu ( d ) \, 2^\frac{n}{d}. 
\end{align}
\end{lem}

As previously mentioned, the statement of Lemma~\ref{lem:neck2LynAper} cannot be generalized to the case of the three-letter alphabet $\mathcal{A}_3$ in a straightforward manner. 
\begin{lem}
\label{lem:neck3LynAper}
Given $n \in \mathbb{N}$, the number of the permuted Lyndon necklaces (the group $C_n^\Pi$) with period $n$ on the set of all words of length $n$ made with the alphabet $\mathcal{A}_3$ is
\begin{align}
\label{eq:neck3LynAper}
N\!L_{\mathcal{A}_3}^\pi(n) = \frac{1}{2n} \bigg[ \sum_{d|n, \, d \, \text{odd}} \mu (d) \, 3^\frac{n}{d} + X_{N\!L}(n)  \bigg],
\end{align}
where 
\begin{align*}
X_{N\!L}(n) = 
\begin{cases}
1, & n=1, \\ 
-1, & n=2^\alpha, \, \alpha \in \mathbb{N}, \\
0, & \text{otherwise}.
\end{cases}
\end{align*}
\end{lem}

\begin{proof}
We directly apply M\"{o}bius inversion formula (Theorem~\ref{thm:mobius}) to~\eqref{eq:neck3Aper} in an adjusted form
\begin{align*}
N^\pi_{\mathcal{A}_3}(n) = 
\frac{1}{2n} \bigg[ \sum_{d|n} \varphi(d) \, 3^{\frac{n}{d}}  + \sum_{d|n, \, d \, \text{even}} \varphi(d) \, 3^\frac{n}{d} + \sum_{d|n, \, d \, \text{odd}} \varphi(d) \bigg]. 
\end{align*}
Thanks to the linearity of M\"obius inversion formula, we may threat the expression summand-wise. 
For the sake of simplicity, the first summand is readily rewritten in the virtue of Lemma~\ref{lem:neckLyn} 
\begin{align*}
N\!L_{\mathcal{A}_3}^\pi (n)  = & \sum_{m|n} \mu(m) \, N_{\mathcal{A}_3}^\pi \bigg( \frac{n}{m} \bigg), \\
 = & \frac{1}{2n} \sum_{d|n} \mu(d) \, 3^\frac{n}{d} + 
\sum_{m|n} \mu(m) \, \frac{m}{2n} \sum_{l|n/m, \, l \, \text{even}} \varphi(l) \, 3^\frac{n}{m l} + \\
& + \sum_{d|n} \mu\left( d \right) \frac{d}{2n} \sum_{l|n/d, \, l \, \text{odd}} \varphi(l). 
\end{align*}

We now want to show that 
\begin{align*}
\sum_{m|n} \mu(m) \, \frac{m}{2n} \sum_{l|n/m, \, l \, \text{even}} \varphi(l) \, 3^\frac{n}{m l} = - \frac{1}{2n} \sum_{d|n, \, d \, \text{even}} \mu(d) \, 3^\frac{n}{d}, 
\end{align*}
which proves the first part of~\eqref{eq:neck3LynAper}. 
Indeed, the use of the substitution $d = ml$ in the virtue of the proof of Lemma~\ref{lem:neckLyn} and~\eqref{eq:convEven} yields
\begin{multline*}
\sum_{m|n} \mu(m) \, \frac{m}{2n} \sum_{l|n/m, \, l \, \text{even}} \varphi(l) \, 3^\frac{n}{m l} = 
\frac{1}{2n} \sum_{d|n} 3^\frac{n}{d} \sum_{l| n/d, \, l \, \text{even}}  \mu \bigg( \frac{d}{l} \bigg) \, \frac{d}{l} \, \varphi(l) =\\= - \frac{1}{2n} \sum_{d|n, \, d \, \text{even}} \mu(d) \, 3^\frac{n}{d}. 
\end{multline*}

The rest of the proof is concluded by the evaluation of  
\begin{align*}
\sum_{d|n} \mu\left( d \right) \frac{d}{2n} \sum_{l|n/d, \, l \, \text{odd}} \varphi(l). 
\end{align*}
Any number $m \in \mathbb{N}$ can be expressed as $m = 2^\alpha P$, where $\alpha \in \mathbb{N}_0$ and $P$ is a product of odd primes. Then 
\begin{align}
\label{eq:neck3LynAperPf1}
\frac{1}{m} \sum_{d/m, \, d \, \text{odd}} \varphi(d) = \frac{P}{m} = \frac{1}{2^\alpha}
\end{align}
since 
\begin{align*}
\sum_{d|m} \varphi(d) = m,
\end{align*}
holds in general,~\cite{Apostol1976}. 

Assume now that $n \in \mathbb{N}$ can be expressed as $n = 2^\beta Q$, where $\beta \in \mathbb{N}_0$ and $Q$ is a product of odd primes. Let us turn our attention to the equality 
\begin{align*}
\frac{1}{n} X_{N\!L(n)} &=  \sum_{d|n} \mu\left( d \right) \frac{d}{n} \sum_{l|n/d, \, l \, \text{odd}} \varphi(l).
\end{align*}
Any $d|n$ can be represented as $2^\gamma R$, where $0 \leq \gamma \leq \beta$ and $R$ is a product of odd primes. 
Decomposing the expression by the exponent $\gamma$ and using~\eqref{eq:neck3LynAperPf1} lead to 
\begin{align}
\label{eq:neck3LynAperPf2}
\frac{1}{n} X_{N\!L(n)} &= \sum_{\gamma = 0} ^ \beta \sum_{R|Q} \mu ( 2^\gamma R ) \, \frac{1}{2^{\beta-\gamma}} = \sum_{\gamma = 0} ^ \beta  \frac{1}{2^{\beta-\gamma}} \sum_{R|Q} \mu ( 2^\gamma R ).
\end{align}

Assume that $\beta = 0$ and $Q=1$. A straightforward computation gives
\begin{align*}
\frac{1}{n} X_{N\!L(n)} \biggr\rvert_{n=1} = \mu(1) \cdot 1 = 1.
\end{align*}

Assume that $\beta > 1$ and $Q=1$. 
If we consider $\gamma \geq 2$ in~\eqref{eq:neck3LynAperPf2}, then $\mu(2^\gamma R) = 0$. 
The sum can be now evaluated 
\begin{align*}
\frac{1}{n} X_{N\!L(n)} \biggr\rvert_{n= 2^\beta} = \mu(1) \,  \frac{1}{2^\beta} + \mu(2) \, \frac{1}{2^{\beta-1}} = -\frac{1}{2^\beta} = -\frac{1}{n}. 
\end{align*}

Assume that $Q > 1$. 
Let us fix $\gamma$ such that $0 \leq \gamma \leq \beta$. 
Without loss of generality, we can assume that $\gamma \leq 1$ and each prime factor in $R$ is present at most once. 
Indeed, $\mu(2^\gamma R) = 0$ otherwise. 
The sign of the nonzero expression $\mu(2^\gamma R)$ is now dependent on the number of prime factors of $R$. 
If there are $m$ prime factors in $Q$, then $R$ with $l$ factors can be chosen in $\binom{m}{l}$ possible ways. 
The sign of $\mu(2^\gamma R)$ alternates as $l$ increases and we have 
\begin{align*}
\sum_{l=0}^m (-1)^l \binom{m}{l} = 0.
\end{align*}
This results in 
\begin{align*}
\frac{1}{n} X_{N\!L(n)} \biggr\rvert_{n= 2^\beta Q} = 0. \tag*{\qedhere}
\end{align*}
\end{proof}

We conclude this section with two lemmas that are direct consequences of M\"obius inversion formula (Theorem~\ref{thm:mobius}), Lemma~\ref{lem:brac2Aper} (respectively \ref{lem:brac3Aper}) and Lemma~\ref{lem:neckLyn}.

\begin{lem}
\label{lem:brac2LynAper}
Given $n \in \mathbb{N}$, the number of the permuted Lyndon bracelets (the group $D_n^\Pi$) with period $n$ on the set of all words of length $n$ made with the alphabet $\mathcal{A}_2$ is
\begin{align}
\label{eq:brac2LynAper}
B\!L_{\mathcal{A}_2}^\pi(n) = \frac{1}{2} \bigg[N\!L_{\mathcal{A}_2}^\pi(n) + \sum_{d|n} \mu\left( \frac{n}{d}\right) \, X^\pi_{B,2}(d) \bigg],  
\end{align}
where $N\!L_{\mathcal{A}_2}^\pi(n)$ and $X^\pi_{B,2}(d)$ are given by~\eqref{eq:neck2LynAper} and ~\eqref{eq:brac2_auxil}, respectively.
\end{lem}

\begin{lem}
\label{lem:brac3LynAper}
Given $n \in \mathbb{N}$, the number of the permuted Lyndon bracelets (the group $D_n^\Pi$) with period $n$ on the set of all words of length $n$ made with the alphabet $\mathcal{A}_3$ is
\begin{align}
\label{eq:brac3LynAper}
B\!L_{\mathcal{A}_3}^\pi(n) = \frac{1}{2} \bigg[N\!L_{\mathcal{A}_3}^\pi(n) + \sum_{d|n} \mu\left( \frac{n}{d}\right) \, X^\pi_{B,3}(d) \bigg],  
\end{align}
where $N\!L_{\mathcal{A}_3}^\pi(n)$ and $X^\pi_{B,3}(d)$ are given by~\eqref{eq:neck3LynAper} and ~\eqref{eq:brac3_auxil}, respectively.
\end{lem}

\section{Conclusion}

We start the final part of this paper with the proof of the main result, Theorem~\ref{thm:main}. 
\begin{proof}[Proof of Theorem~\ref{thm:main}]
For given $n \geq 2$ the sequence $B\!L_{\mathcal{A}_3}^\pi(n)$ gives the number of the permuted Lyndon bracelets, i.e., the equivalence classes of words with respect to the rotations $r^i$, the reflection $s$, the value permutation $\pi$ (group $D_n^\Pi$), their compositions and with primitive period of length $n$. 
As discussed in \S\ref{sec:syms}, regions $\Omega_\mathrm{w}$ surely have identical (the rotations $r_i$, the reflections $s$) or similar, with respect to the operator $\mathcal{T}$ defined by~\eqref{eq:T}, (the value permutation $\pi$) shape and are thus qualitatively equivalent (see Definition~\ref{def:regions}).  
The expression~\eqref{eq:brac3LynAper_main} is exactly~\eqref{eq:brac3LynAper} with~\eqref{eq:neck3LynAper} substituted and~\eqref{eq:brac-auxil-intro} corresponds to~\eqref{eq:brac3_auxil}. 
We sum $B\!L_{\mathcal{A}_3}^\pi(n)$ from $m=2$ to avoid including trivial existence regions for the constant words $\mathfrak{00}\ldots\mathfrak{0}$, $\mathfrak{aa}\ldots\mathfrak{a}$ and $\mathfrak{11}\ldots\mathfrak{1}$ which are represented by the additional $1$, this yields~\eqref{eq:total3_intro}. 
The formulas are upper estimates only since we cannot eliminate the possibility that there are two qualitatively equivalent regions whose respective words are not related by any of the symmetries. 
Similar argumentation holds for regions belonging to the stable stationary solutions since the corresponding words are made with the alphabet $\mathcal{A}_2$, Lemma~\ref{lem:equiv} and Definition~\ref{def:sol-type}. 
The expression~\eqref{eq:brac2LynAper_main} is equal to~\eqref{eq:brac2LynAper} where $B\!L_{\mathcal{A}_2}^\pi(n)$ is given by~\eqref{eq:neck2LynAper}. \qedhere
\end{proof}

The approach presented here can be used to obtain similar results in other or more general settings. 
The two main extension directions are the change of a spatial structure and the change of dynamics.
The extensions can be combined but we present them separately for the sake of clarity.

\subsection{Change of spatial structure}
\label{sec:spat_str_change}
\subsubsection{Graphs with nontrivial automorphism}
The main objects of interests were the LDE~\eqref{eq:LDE} and the GDE~\eqref{eq:GDE} in this paper. 
In general, given a graph $\mathcal{G} = (V,E)$, the Nagumo graph differential equation can be written as
\begin{align*}
\mathrm{u}_i'(t) = d \sum_{j \in \mathcal{N}(i)} \big(\mathrm{u}_j(t) - \mathrm{u}_i(t) \big) + f\big(\mathrm{u}_i(t);a\big),
\end{align*}
where $i \in V$ and $\mathcal{N}(i)$ is the set of all neighbours of the vertex $i$, i.e., $j \in \mathcal{N}(i)$ if and only if $(i,j) \in E$. 
Provided the graph $\mathcal{G}$ has a nontrivial automorphism (a nontrivial self-map which preserves the edge-vertex connectivity) the approach used here can be extended. 
Indeed, the group $D_n$ is the automorphism group of the cycle graph with $n$ vertices and all computation can be carried out by replacing the dihedral group $D_n$ with the automorphism group of the graph $\mathcal{G}$. 

\subsubsection{Multi-dimensional square lattices}
The underlying spatial structure of the LDE~\eqref{eq:LDE} is a one-dimensional lattice, an infinite path graph. 
Examination of bistable reaction diffusion systems on multi-dimensional square lattices has been carried out, see e.g.,~\cite{Chen2008,Hupkes2020,Hupkes2013}.
For example, let us have a bistable reaction-diffusion system on the two-dimensional square lattice
\begin{align}
\label{eq:LDE2D}
u'_{i,j}(t) = d \big(u_{i-1,j}(t)+ u_{i+1,j}(t) + u_{i,j-1}(t) + u_{i,j+1}(t) - 4 u_{i,j}(t) \big) + f\big(u_{i,j}(t);a\big), 
\end{align}
for $i,j \in \mathbb{Z}$. 
A reproduction of the proof of Lemma~\ref{lem:equiv} together with the comparison principle~\cite[Proposition 3.1]{Hoffman2017} shows that 
the stationary solutions of the LDE~\eqref{eq:LDE2D} in the form of a repeated $2 \times 2$ pattern are equivalent to the stationary solutions of the GDE~\eqref{eq:GDE} on four vertices with the doubled diffusion rate $d$, see Figure~\ref{fig:2Dlattice} for illustration. 

\begin{figure}[ht]
\centering
\includegraphics[width = \linewidth]{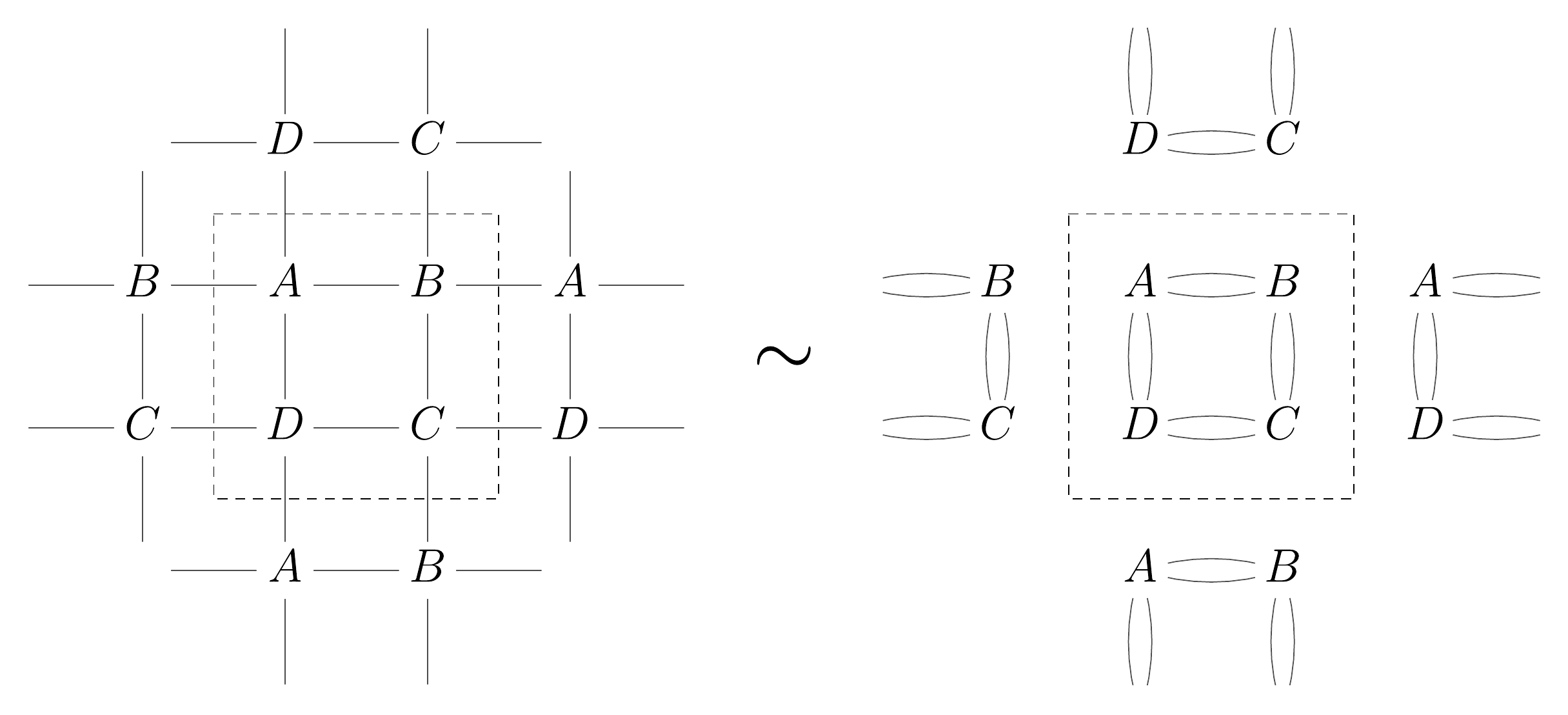}
\caption{Illustration of a possible Lemma~\ref{lem:equiv} extension for patterns on two-dimensional lattices.
A general idea is that the edges crossing the dashed line are wrapped back inside from the opposite sides. }
\label{fig:2Dlattice}
\end{figure}

\subsection{Change of dynamics}
Various changes in the nonlinear part of~\eqref{eq:LDE} are discussed here. 
The proof of Lemma~\ref{lem:equiv} in~\cite[Lemma 1]{Hupkes2019c} is actually independent of the nonlinear term with one exception. 
Indeed, the only part of the proof dependent on the specific nonlinear term is the comparison principle from~\cite[Lemma 1]{Chen2008} and the only assumption is the existence of two ordered steady states of the equation, constant $0$ and constant $1$ in our case. 
This is satisfied for bistable and multistable reaction terms presented here. 

\subsubsection{Scaled cubic nonlinearity} 
The cubic bistable nonlinear term~\eqref{eq:cubic} is dependent on one parameter only and moreover, the value of the parameter is actually one of its roots. 
Let us assume the function
\begin{align}
\label{eq:cubic_gen}
f_\mathrm{cub}(s, \mathrm{p}) = s\big(s-\nu_-(\mathrm{p})\big) \big(\nu_+(\mathrm{p})-s \big),
\end{align}
where $\mathrm{p}\in \Theta \subset \mathbb{R}^m$ is a detuning vector, $\Theta$ is an open set and we assume that $\nu_-,\nu_+ \colon \Theta \to \mathbb{R}^+$ and $0< \nu_-(\mathrm{p}) < \nu_+(\mathrm{p})$ for all $\mathrm{p} \in \Theta$. 
The term $f_\mathrm{cub}$ has two bounding roots $0$ and $\nu_+(\mathrm{p})$ for any given $\mathrm{p} \in \Theta$. 
The LDE~\eqref{eq:LDE} with $f_\mathrm{cub}$ thus admits the comparison principle and its $n$-periodic stationary solutions correspond to the stationary solutions of its respective GDE on a cycle graph with $n$ vertices. 

The stationary problem for the GDE can be written in the form
\begin{align}
\label{eq:GDEstatScaled}
\widetilde{h}(\mathrm{u};\mathrm{p},d)=0,
\end{align}
where
\begin{align*}
\widetilde{h}_i(\mathrm{u};\mathrm{p},d) =
d(\mathrm{u}_{i-1} -2 \mathrm{u}_i + \mathrm{u}_{i+1}) + f_\mathrm{cub}(\mathrm{u}_i, \mathrm{p}). 
\end{align*}
We omitted the modulo wrapping at vertices $1$ and $n$ as in~\eqref{eq:GDEstat} to enlighten the notation. 
A direct computation yields
\begin{align*}
\widetilde{h}(\mathrm{u};\mathrm{p},d) = \nu_+^3(\mathrm{p}) \, h\bigg( \frac{\mathrm{u}}{\nu_+(\mathrm{p})}; \frac{\nu_-(\mathrm{p})}{\nu_+(\mathrm{p})}, \frac{d}{\nu_+^2(\mathrm{p})} \bigg). 
\end{align*}
This enables us to define solution types for~\eqref{eq:GDEstatScaled} via Definition~\ref{def:sol-type} (the sign of the first derivative's determinant agrees) and to obtain corresponding existence regions $\widetilde{\Omega}_\mathrm{w}$ through the implicit transformation
\begin{align}
\label{eq:scaledTransform}
\widetilde{\Omega}_\mathrm{w} = \Bigg\{ (\mathrm{p},d) \in \Theta \times \mathbb{R}_0^+ \, \bigg| \, \bigg( \frac{\nu_-(\mathrm{p})}{\nu_+(\mathrm{p})}, \frac{d}{\nu_+^2(\mathrm{p})} \bigg) \in \Omega_\mathrm{w} \Bigg\}.
\end{align}

An example of system leading to~\eqref{eq:GDEstatScaled} is the reduced version of the model decribing potential propagation in myelinated axon with recovery~\cite{Bell1981}
\begin{align}
\begin{split}
\label{eq:LDEBell}
u_i'(t) &= d\big( u_{i-1}(t) - 2u_i(t) + u_{i+1}(t) \big) + f_\text{Bell} \big(u_i(t); a, b \big) - v_i(t), \\
v_i'(t) &= \sigma \, u_i(t) - \gamma \, v_i(t)
\end{split}
\end{align}
such that 
\begin{align*}
f_\text{Bell} \big(s; a, b \big) = s(s-a)(b-s). 
\end{align*}
Via approach similar to~\cite{Bell1981}, we assume, that the change of the recovery value $v_i$ is faster than the change in $u_i$ and thus the second equation in~\eqref{eq:LDEBell} resides at its steady state. 
The problem can be then expressed as 
\begin{align*}
u'_i(t) = d\big( u_{i-1}(t) - 2u_i(t) + u_{i+1}(t) \big) + f_\text{Bell}\big(u_i(t); a, b\big) - \beta \, u_i(t)
\end{align*}
with $\beta = \sigma/\gamma$ possibly small and the generalization of Lemma~\ref{lem:equiv} ensures the equivalence of the periodic steady states of~\eqref{eq:LDEBell} and the system \eqref{eq:GDEstatScaled} solutions.
We can directly determine
\begin{align}
\mathbf{p} = (a,b,\beta), \quad \Theta &= \bigg( (a,b,\beta) \in \mathbb{R}^3 \, \bigg| \, a,b,\beta>0, \, b>a, \, \beta< \frac{(a-b)^2}{4}\bigg), \\
\nu_\pm(a,b,\beta) &= \frac{1}{2} \Big( a+b \pm \sqrt{(a-b)^2-4\beta} \Big).
\end{align} 
The inequality $b>a$ preserves the bistable behaviour in the original equation. 
See Figure~\ref{fig:Bell} for illustration. 

\begin{figure}[ht]
\begin{subfigure}[b]{.43\textwidth}
\centering
\includegraphics[width=\textwidth]{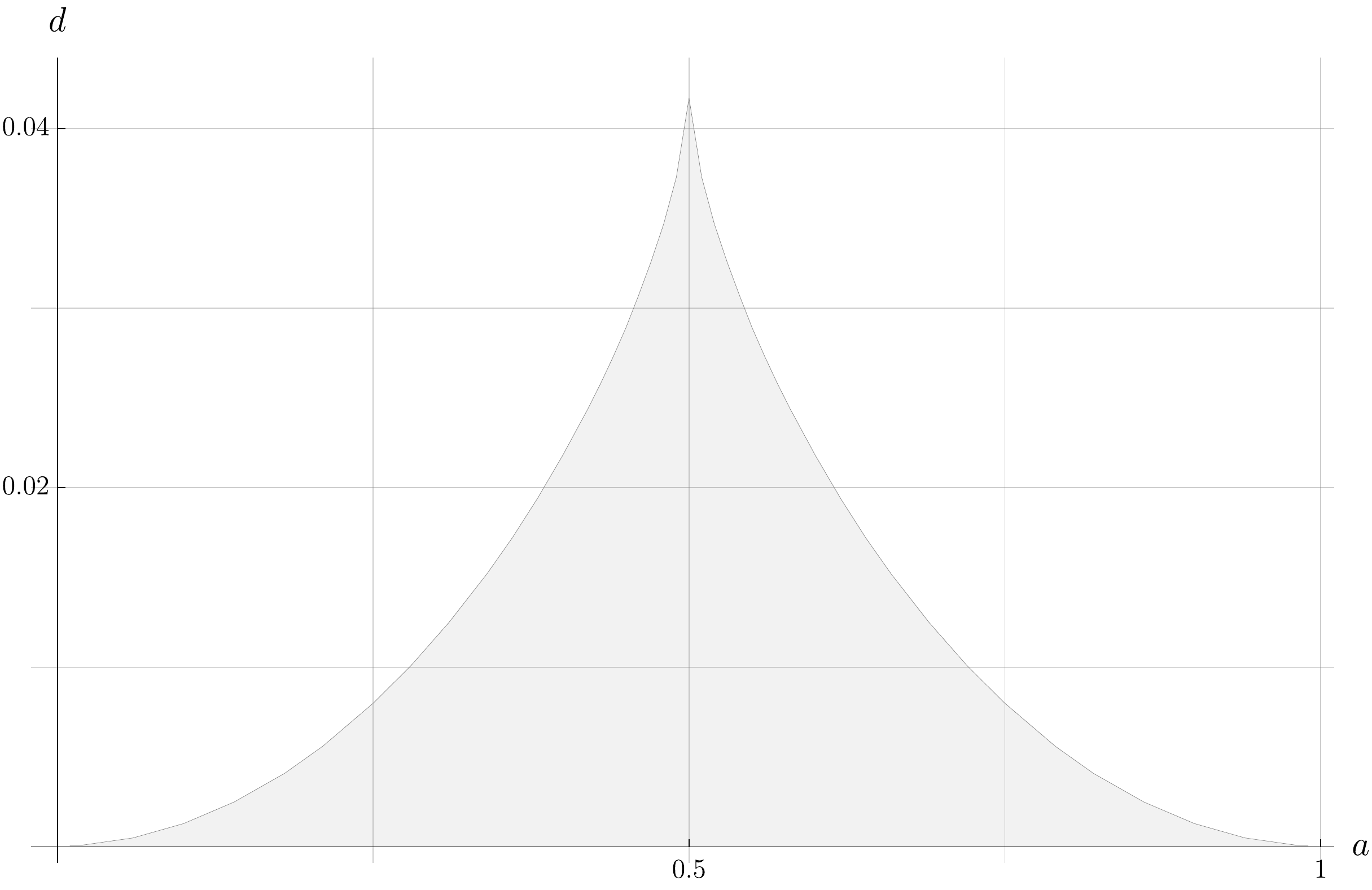}
\end{subfigure}
\hspace{3pt}
\begin{subfigure}[b]{.53\textwidth}
\centering
\includegraphics[width=\textwidth]{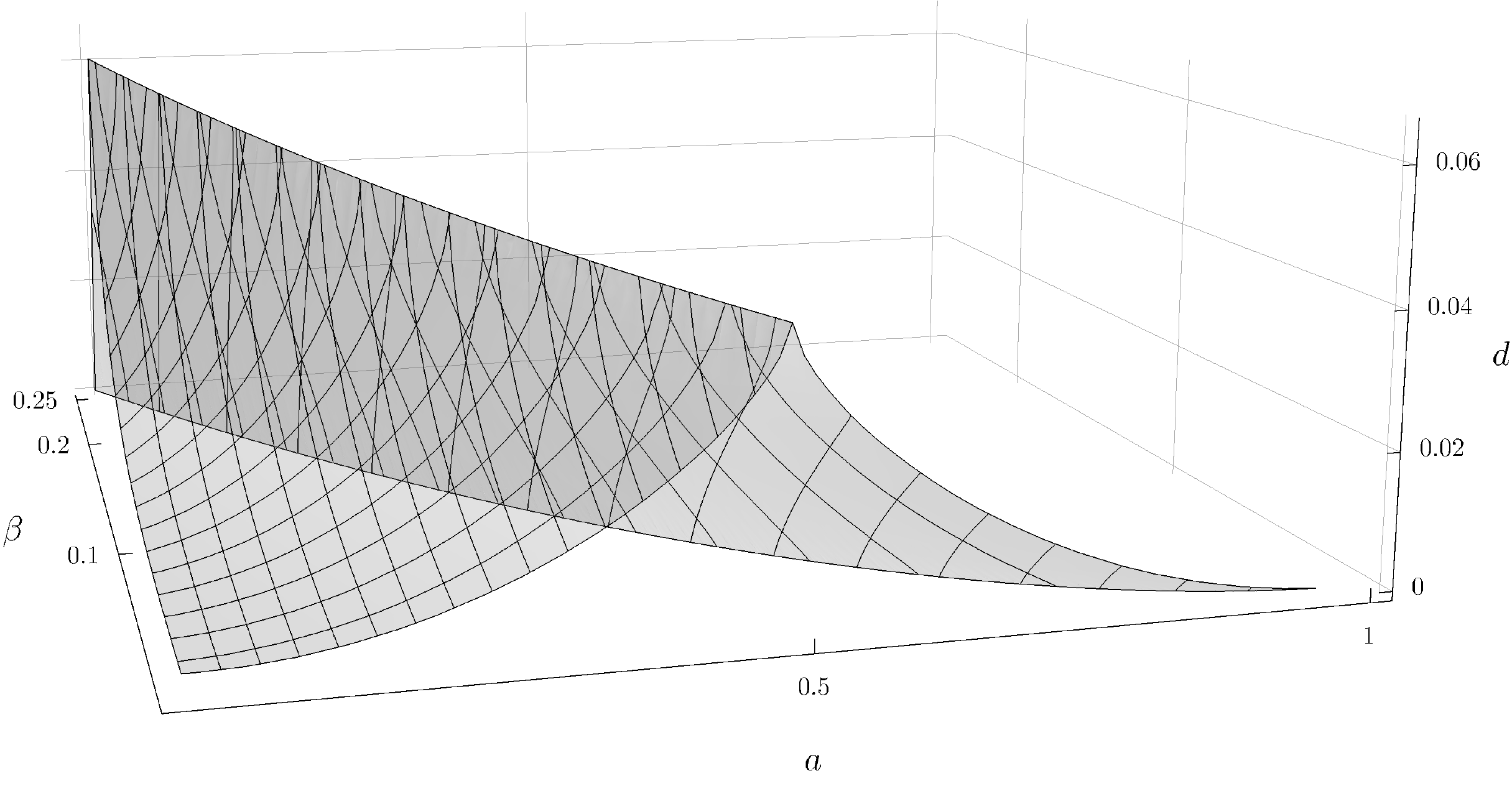}
\end{subfigure}
\caption{The left panel depicts the region $\Omega_\mathfrak{01}$ for the equation~\eqref{eq:LDE} and its comparison to the same region for~\eqref{eq:LDEBell} obtained via the transformation in~\eqref{eq:scaledTransform}. 
The parameter $b=1$ was set. }
\label{fig:Bell}
\end{figure}

\subsubsection{Polynomial nonlinearity of higher order}
This paper focused on the model~\eqref{eq:LDE} with the cubic bistable nonlinearity
\begin{align*}
f(s;a) = s(1-s)(s-a). 
\end{align*}
The idea presented in \S\ref{ssec:prel:ss} can be extended to a general polynomial nonlinearity provided it allows a spatially nonhomogeneous steady state of the LDE~\eqref{eq:LDE} or the GDE~\eqref{eq:GDE}
\begin{align*}
f_{\text{ext}}(s; a_1, \ldots, a_q) = s(1-s) \prod_{i=1}^q (s-a_i)
\end{align*}
for $q \geq 3$ odd, $a_i \in (0,1)$ and $a_i \neq a_j$ for all $i,j \in \{1, \ldots, q\}$ such that $i \neq j$. 
Note that
\begin{align*}
f_{\text{ext}}(s; a_1, \ldots, a_q) = -f_{\text{ext}}(1-s; 1-a_1, \ldots, 1-a_q)
\end{align*} holds and the value permutation $\pi$ can be thus redefined as 
\begin{align*}
\big(\pi (\mathrm{w}) \big)_i = 
\begin{cases}
\mathfrak{1}, & \mathrm{w}_i = \mathfrak{0}, \\
\mathfrak{a}_{q-i+1}, & \mathrm{w}_i = \mathfrak{a}_i, \, i=1,\ldots,q, \\
\mathfrak{0}, & \mathrm{w}_i = \mathfrak{1}.
\end{cases}
\end{align*}
The counting formulas for the necklaces~\eqref{eq:neck}, the bracelets~\eqref{eq:brac} and the Lyndon words~\eqref{eq:neckLyn},~\eqref{eq:bracLyn} can be then straightforwardly applied with $k=q+2$ for all solutions and $k=2+(q-1)/2$ for asymptotically stable solutions. 

The cubic-quintic nonlinearity,~\cite{Boudebs2003}, 
\begin{align*}
f_\text{cq}(s,\mu) = \mu s + 2 s^3 - s^5   
\end{align*}
has five distinct roots
\begin{align*}
\Big\{ 0, \pm \sqrt{1\pm\sqrt{1-\mu}} \Big\}. 
\end{align*}
for $\mu \in (0,1)$. 
The LDE with $f_\text{cq}$ can be rescaled for the stationary solutions to fit the interval $[0,1]$ and the approach described in the previous paragraph can be used. 
Note that sequence of $1/2$'s is then always a stationary solution regardless of $\mu$ and all the counting formulas would count not only shape-distinct regions $\Omega_\mathrm{w}$ but distinct periodic stationary solutions. 
This is true for~\eqref{eq:LDE} only if $a=1/2$.

\subsubsection{General bistable nonlinearity}
A system with a general bistable nonlinearity $f_{\text{gen}}$ as considered in~\cite{Keener1987} 
\begin{enumerate}
\item 
$f_{\text{gen}}(0) = f_{\text{gen}}(a) = f_{\text{gen}}(1)=0$, $0 < a < 1$ and $f_{\text{gen}}(x) \neq 0$ for $x \neq 0, a, 1$, 
\item
$f_{\text{gen}}(x) < 0 $ for $0<x<a$ and $f_{\text{gen}}(x)>0$ for $a < x <1$, 
\item 
$f'_{\text{gen}}(x_0) = f'_{\text{gen}}(x_1) = 0$, $0< x_0 < a < x_1 < 1$ and $f'_{\text{gen}}(x) \neq 0$ for $x \neq x_0, x_1$, 
\end{enumerate}
can be only partially treated by the methods presented here. 
The conditions above allow the application of the implicit function theorem. 
It is however possible for a general bistable nonlinearity to exhibit the ``blue sky'' bifurcation before any of the determinants in Definition~\ref{def:sol-type} reaches zero, see~\cite[\S1.2.2]{Morelli2019}. 
Moreover, the action of the value permutation group $\Pi$ can be included only if $f_{\text{gen}}$ can be expressed in a form such that
\begin{align*}
f_{\text{gen}}(x;a) = -f_{\text{gen}}(1-x;1-a)
\end{align*}
holds. 

\subsubsection{Multi-dimensional local dynamics}
The local dynamics at an isolated vertex of models~\eqref{eq:LDE} and~\eqref{eq:GDE} are one-dimensional since the behaviour at a single vertex can be described by a single equation. 
This is not always the case in many reaction-diffusion models. 
For example, the Lotka-Volterra competition model on a graph as in~\cite{Slavik2020}
\begin{align}
\label{eq:lv}
\begin{split}
u'_i(t) &= d_u \sum_{j \in N(i)} \big(u_j(t)-u_i(t) \big) + \rho_u u_i(t) \big(1-u_i(t)-\alpha v_i(t) \big), \\
v'_i(t) &= d_v \sum_{j \in N(i)} \big(v_j(t)-v_i(t) \big) + \rho_v v_i(t) \big(1-\beta u_i(t)- v_i(t) \big), 
\end{split}
\end{align}
where $N(i)$ is the set of all neighbours of the vertex $i$, 
locally possesses two asymptotically stable stationary solutions (originating from the points $(0,1)$ and $(1,0)$ which can be denoted by $\mathfrak{0},\mathfrak{1}$) and one unstable nontrivial stationary solution (originating from the point $((1-\alpha)/(1-\alpha\beta),(1-\beta)/(1-\alpha\beta))$ here denoted by $\mathfrak{a}$) at each separated vertex provided $\alpha, \beta > 1$. 
The solutions containing elements originating from $(0,0)$ are not considered since their immediate continuation is directed outside the positive quadrant. 
The implicit function theorem assures that the naming scheme from Definition~\ref{def:sol-type} can be employed. 
A proper scaling results in $\rho_u = \rho_v = 1$ and the regions of existence are pathwise connected sets of points $(d_u,d_v,\alpha,\beta) \in \mathbb{R}_0^+ \times \mathbb{R}_0^+ \times (1,\infty) \times (1,\infty)$.
As in~\cite{Slavik2020}, it is convenient to fix the ratio $\eta := d_u/d_v$ and consider the regions $\Omega_\mathrm{w}$ in a three-dimensional space only. 
The stationary problem for~\eqref{eq:lv} is now invariant with respect to the transformation $u_i \leftrightarrow v_i$, $\alpha \leftrightarrow \beta$ and all counting results can be thus applied provided the underlying graph has a nontrivial automorphism.

\subsection{Open Questions}
The idea of Lemma~\ref{lem:equiv} is such that the restriction to the periodic stationary solutions of the LDE~\eqref{eq:LDE} allows us to formally divide the lattice into a countable number of identical finite graphs. 
Similar approach was used in \S\ref{sec:spat_str_change}. 
General equivalence claim which helps to reduce the search for an arbitrary periodic patterns in sufficiently regular infinite graphs (e.g., triangular lattice, hexagonal lattice) into a finite-dimensional problem is still missing.


\begin{thebibliography}{99}

\bibitem{Allen1987}
\textsc{L. J. S. Allen}, Persistence, extinction, and critical patch number for island populations,
\textit{J. of Math. Biol.} \textbf{24}(1987), No.~6, 617--625.

\bibitem{Apostol1976}
\textsc{T. M. Apostol}, Introduction to analytic number theory, Springer-Verlag, New York, 1976.

\bibitem{Beeler1977}
\textsc{G. W. Beeler, H. Reuter}, Reconstruction of the action potential of ventricular myocardial fibres,
\textit{J. Physiol.} \textbf{268}(1977), No.~1, 177--210. 

\bibitem{Bell1981}
\textsc{J. Bell}, Some threshold results for models of myelinated nerves,
\textit{Math. Biosci.} \textbf{54}(1981), No.~3--4, 181--190. 

\bibitem{Boudebs2003}
\textsc{G. Boudebs, S. Cherukulappurath, H. Leblond, J. Troles, F. Smektala, F. Sanchez}, Experimental and theoretical study of higher-order nonlinearities in chalcogenide glasses,
\textit{Opt. Commun} \textbf{219}(2003), No.~1--6, 427--433. 

\bibitem{Bramburger2018}
\textsc{J. J. Bramburger}, Rotating wave solutions to lattice synamical systems I: The anti-continuum limit,
\textit{J. Dyn. Differ. Equ.} \textbf{31}(2019), No.~1, 469--498. 

\bibitem{Bramburger2019}
\textsc{J. J. Bramburger, B. Sandstede}, Spatially localized structures in lattice dynamical systems,
\textit{J. Nonlinear Sci.} \textbf{30}(2020), No.~2, 603--644. 

\bibitem{Burnside1911}
\textsc{W. Burnside}, Theory of groups of finite order, Cambridge University Press, 1911. 

\bibitem{Butler1991}
\textsc{G. Butler}, Fundamental algorithms for permutation groups, Springer Berlin
Heidelberg, 1991. 

\bibitem{Chen2008}
\textsc{X. Chen, J.-S. Guo, C.-C. Wu.}, Traveling waves in discrete periodic media for bistable dynamics,
\textit{Arch. Ration. Mech. An.} \textbf{189}(2008), 189--236. 

\bibitem{Cheng2005}
\textsc{C.-Y. Cheng, C.-W. Shih}, Pattern formations and spatial entropy for spatially discrete diffusion equations,
\textit{Physica D} \textbf{204}(2005), No.~3--4, 135--160. 

\bibitem{Fine1958}
\textsc{N. J. Fine}, Classes of periodic sequences,
\textit{Illinois J. Math.} \textbf{2}(1958), No.~2, 285--302. 

\bibitem{Gilbert1961}
\textsc{E. N. Gilbert, J. Riordan}, Symmetry types of periodic sequences,
\textit{Illinois J. Math.} \textbf{5}(1961), No.~4, 657--665. 

\bibitem{Golomb1958}
\textsc{S. W. Golomb, B. Gordon, L. R. Welch}, Comma-free codes,
\textit{Can. J. Math.} \textbf{10}(1993), 202--209. 

\bibitem{Hoffman2017}
\textsc{A. Hoffman, H. Hupkes, E. V. Vleck}, Entire solutions for bistable lattice differential equations with obstacles,
\textit{Mem. Am. Math. Soc.} \textbf{250}(2017), No.~1188. 

\bibitem{Hosek2019}
\textsc{R. Ho\v{s}ek and J. Volek}, Discrete advection-diffusion equations on graphs: maximum principle and finite volumes,
\textit{Appl. Math. Comput.} \textbf{361}(2019), 630--644. 

\bibitem{Hupkes2020}
\textsc{H. J. Hupkes, L.Morelli}, Travelling corners for spatially discrete reaction-diffusion systems,
\textit{Commun. Pure Appl. Anal.} \textbf{19}(2020), No.~3, 1609--1667. 

\bibitem{Hupkes2019a}
\textsc{H. J. Hupkes, L. Morelli, P. Stehl\'{i}k}, Bichromatic travelling waves for lattice Nagumo equations,
\textit{SIAM J. Appl. Dyn. Syst.} \textbf{18}(2019), No.~2, 973--1014. 

\bibitem{Hupkes2019b}
\textsc{H. J. Hupkes, L. Morelli, P. Stehl\'{i}k, V. \v{S}v\'{i}gler}, Multichromatic travelling waves for lattice Nagumo equations,
\textit{Appl. Math. Comput.} \textbf{361}(2019), 430--452. 

\bibitem{Hupkes2019c}
\textsc{H. J. Hupkes, L. Morelli, P. Stehl\'{i}k, V. \v{S}v\'{i}gler}, Counting and ordering periodic stationary solutions of lattice Nagumo equations,
\textit{Appl. Math. Lett.} \textbf{98}(2019), 398--405. 

\bibitem{Hupkes2013}
\textsc{H. J. Hupkes, E. S. Van Vleck}, Negative diffusion and traveling waves in high dimensional lattice systems,
\textit{SIAM J. Math. Anal.} \textbf{45}(2013), No.~3, 1068--1135. 

\bibitem{Hupkes2015}
\textsc{H. J. Hupkes, E. S. Van Vleck}, Travelling waves for complete discretizations of reaction diffusion systems,
\textit{J. Dyn. Differ. Equations} \textbf{28}(2016), No.~3--4, 955--1006. 

\bibitem{Keener1987}
\textsc{J. P. Keener}, Propagation and its failure in coupled systems of discrete excitable cells,
\textit{SIAM J. Appl. Math.} \textbf{47}(1987), No.~3, 556--572. 

\bibitem{Kevrekidis2009}
\textsc{P. G. Kevrekidis}, The discrete nonlinear Schrödinger equation,
Springer, 2009. 

\bibitem{Kiss2017}
\textsc{I. Z. Kiss, J. C. Miller, P. L. Simon}, Mathematics of epidemics on networks. From exact to approximate models,
Springer, 2017. 

\bibitem{Laplante1992}
\textsc{J. Laplante, T. Erneux.}, Propagation failure and multiple steady states in an array of diffusion coupled flow reactors,
\textit{Physica A} \textbf{188}(1992), No.~1--3, 89--98. 

\bibitem{Levin1974}
\textsc{S. A. Levin}, Dispersion and population interactions,
\textit{Am. Nat.} \textbf{108}(1974), No.~960, 207--228. 

\bibitem{Lindeberg1990}
\textsc{T. Lindeberg}, Scale-space for discrete dignals,
\textit{IEEE T. Pattern Anal.} \textbf{12}(1990), No.~3, 234--254. 

\bibitem{Mobius1832}
\textsc{A M\"{o}bius}, Über eine besondere Art von Umkehrung der Reihen (in German) [About a special kind of reversal of the series],
\textit{J. Reine Angew. Math} \textbf{9}(1993), 105--123.

\bibitem{Morelli2019}
\textsc{L. Morelli}, Travelling patterns on discrete media, PhD Thesis, Leiden University, 2019.

\bibitem{Nagumo1962}
\textsc{J. Nagumo, S. Arimoto, S. Yoshizawa}, An active pulse transmission line simulating nerve axon,
\textit{Proc. IRE} \textbf{50}(1962), No.~10, 2061--2070. 

\bibitem{Nagumo1965}
\textsc{J. Nagumo, S. Yoshizawa, S. Arimoto}, Bistable transmission lines,
\textit{IEEE T. Circuit Th.} \textbf{12}(1965), No.~3, 400--412. 

\bibitem{Riordan1958}
\textsc{J. Riordan}, An introduction to combinatorial analysis,  John Wiley \& Sons, Inc., 1958.

\bibitem{Slavik2020}
\textsc{A. Slav\'{i}k}, Lotka-Volterra competition model on graphs,  \textit{SIAM J. Appl. Dyn. Syst.} \textbf{19}(2020), No.~2, 725--762. 

\bibitem{Stehlik2017}
\textsc{P. Stehl\'{i}k}, Exponential number of stationary solutions for Nagumo equations on graphs,  \textit{J. Math. Anal. Appl.} \textbf{455}(2017), No.~2, 1749--1764. 

\bibitem{Vleck1998}
\textsc{J. W. Cahn, J. Mallet-Paret, E. S. Van Vleck}, Traveling wave solutions for systems of ODEs on a two-dimensional spatial lattice,  \textit{SIAM J. Appl. Math.} \textbf{59}(1999), No.~2, 455--493. 

\bibitem{Volek2016}
\textsc{J. Volek}, Landesman-Lazer conditions for difference equations involving sublinear perturbations,  \textit{J. Difference Equ. Appl.} \textbf{22}(2016), No.~11, 1698--1719. 


\bibitem{Zinner1992}
\textsc{B. Zinner}, Existence of traveling wavefront solutions for the discrete Nagumo equation,  \textit{J. Differ. Equations} \textbf{96}(1992), No.~1, 1--27. 
\end{thebibliography}
\end{document}